\documentclass[12pt]{amsart}
\usepackage{amsmath}
\usepackage{amsfonts}
\usepackage{hyperref}
\usepackage{amssymb}
\usepackage{latexsym}
\usepackage{amsthm}
\usepackage{a4wide}

\newtheorem{theorem}{Theorem}
\newtheorem{lemma}[theorem]{Lemma}

\newtheorem{coro}[theorem]{Corollary}

\newtheorem{proposition}[theorem]{Proposition}

\newcounter{other}            
\newtheorem{otherth}[other]{Theorem}              
\newtheorem{otherl}[other]{Lemma}        

\theoremstyle{definition}
\newtheorem{definition}[theorem]{Definition}

\newcommand{\Cn}{\mathbb{C}^n}

\newcommand{\Sn}{\mathbb{S}_ n}

\newcommand{\Bn}{\mathbb{B}_ n}

\numberwithin{equation}{section}
\begin{document}

\title{Duality of holomorphic Hardy type tent spaces}

\author[Antti Per\"{a}l\"{a}]{Antti Per\"{a}l\"{a}}
\address{Antti Per\"{a}l\"{a} \\Departament de Matem\`{a}tiques i Inform\`{a}tica \\
Universitat de Barcelona\\
08007 Barcelona\\
Catalonia, Spain\\
Barcelona Graduate School of Mathematics (BGSMath).} \email{perala@ub.edu}


%


\thanks{The author acknowledges financial support from the Spanish Ministry of Economy and Competitiveness, through the Mar\'ia de Maeztu Programme for Units of Excellence in R\&D (MDM-2014-0445). The author was partially supported by the grant MTM2017-83499-P (Ministerio de Educaci\'on y Ciencia).}


\begin{abstract}
We study the holomorphic tent spaces $\mathcal{HT}^p_{q,\alpha}(\Bn)$, which are motivated by the area function description of the Hardy spaces on one hand, and the maximal function description of the Hardy spaces on the other. Characterizations for these spaces under general fractional differential operators are given.

We describe the dual of $\mathcal{HT}^p_{q,\alpha}(\Bn)$ for the full range $0<p,q<\infty$ and $\alpha>-n-1$. Here the case $1<p<\infty$, $q\leq 1$ leads us to the Hardy-Bloch type spaces $\mathcal{BT}^{p}(\Bn)$ and $\mathcal{BT}^{p,0}(\Bn)$, the latter being the predual of $\mathcal{HT}^p_{1,\alpha}(\Bn)$. In an analogous fashion, the case $p=1<q<\infty$ gives rise to Hardy-Carleson type spaces $\mathcal{CT}_{q,\alpha}(\Bn)$ and $\mathcal{CT}^0_{q,\alpha}(\Bn)$. In the remaining cases, the duality can be described in terms of the classical Bloch space $\mathcal{B}(\Bn)$. We also study these spaces in detail, providing characterizations in terms of fractional derivatives.

When $p,q\geq 1$, our approach to the question of duality is strongly related to the boundedness of a certain weighted Bergman projections acting on tent spaces. For the small exponents, we will also need to apply some other techniques, such as atomic decompositions and embedding theorems.

The work presented gives a unified approach for the classical Hardy and Bergman dualities, also involving the Bloch and BMOA spaces as well as some new spaces. We end the paper with discussion of some further questions and related topics.
\end{abstract}

\maketitle



\section{Introduction}

\noindent The concept of tent spaces, introduced by R. R. Coifman, Y. Meyer and E. M. Stein \cite{CMS}, arises from harmonic analysis. Related techniques have found application in holomorphic Hardy and Bergman spaces, in connection to Carleson measures, atomic decompositions and various operators acting between these spaces. We mention here \cite{AB, CV, FS, Jev, Lue1, Lue2, P1, PP, PR2015, Pel}. Motivated by the fact that the derivatives of Hardy space functions are completely characterized in terms of their membership in a certain tent space of holomorphic functions, it is indeed natural to study these spaces under the additional assumption of holomorphicity.

From the definition of a generic tent space $T^p_q$, it is not surprising that one has the duality $$(T^p_q)^*\sim T^{p'}_{q'},$$ where $1<p,q<\infty$ and $p'$ and $q'$ are the usual dual exponents. It is absolutely natural to expect such duality result to hold for the holomorphic tent spaces as well. On the other hand, it is well-known (see \cite{GK}, \cite{JPR}) that the analogous result can fail even for the holomorphic Lebesgue spaces (although a very natural duality result exists), and the weights need not be pathological, see also \cite{Hed}. Therefore, it is not completely obvious, what one should expect, even though the end-result might not come as a surprise. Also, it is not clear what to expect when $p,q\leq 1$; typically holomorphicity guarantees rich structure of a dual space, even for these small exponents. The purpose of the present work is to study these questions.

Our first main result is Theorem \ref{duality}, which describes the duality for holomorphic Hardy type tent spaces $\mathcal{HT}^p_{q,\alpha}(\Bn)$, when $1<p<\infty$, effectively obtaining the expected duality
$$(\mathcal{HT}^p_{q,\alpha}(\Bn))^*\sim \mathcal{HT}^{p'}_{q',\alpha}(\Bn),$$
when $1<p,q<\infty$. The case $q\leq 1$ leads to holomorphic Hardy-Bloch type spaces $\mathcal{BT}^p(\Bn)$ that seem do not seem to have appeared in the literature before. We establish
$$(\mathcal{HT}^p_{q,\alpha}(\Bn))^*\sim \mathcal{BT}^{p'}(\Bn).$$
The case $q<1$ is more demanding and is therefore treated separately in Theorem \ref{duality3}. In this case the dual pairing will change in a manner analogous to \cite{Bla}, \cite{PR}, \cite{Sha}, \cite{ZhuSmall}, for instance. 

The case $0<p\leq 1$ gives rise to so-called Hardy-Carleson type spaces $\mathcal{CT}_{q,\alpha}(\Bn)$. This nomenclature is motivated by the fact that these spaces are defined in terms of Carleson measures, in a fashion similar to the way the classical $BMOA(\Bn)$ can be defined in terms of Carleson measures.

Some key points in the proof are contained in Propositions \ref{projection} and \ref{projection2}, which are interesting in their own right. These results establish the boundedness of a certain Bergman type projection on $T^p_{q,\alpha}(\Bn)$. This projection arises from the dual pairing of the measurable tent spaces, and we extend some well-known properties to the tent space case.

Along the way to towards the desired duality result, it well be well-motivated to study the spaces $\mathcal{HT}^p_{q,\alpha}(\Bn)$ and $\mathcal{BT}^p(\Bn)$ is detail. Theorems \ref{HTchar} and \ref{BTchar} characterize these spaces in terms of the fractional derivatives $R^{s,t}$ (see \cite{ZZ, ZhuSmall, ZhuBn}). These theorems are also used to obtain an approximation result in Lemma \ref{approx}, which will be needed to complete the duality. For the case $q<1$, we will also apply atomic decomposition techniques in Lemmas \ref{seq} and \ref{converse}.

The complete description of duality is expressed via the Table \ref{table} below. For a reader, who is familiar with the standard Hardy and Bergman dualities only when $p\geq 1$, it might come as a surprise that when $0<p<1$, the dual of both Hardy and Bergman spaces can be identified with the Bloch space, see \cite{ZhuSmall}. A similar phenomenon (or an extension of this result) occurs here; when $p<1$, the dual space can always be identified with $\mathcal{B}(\Bn)$.

\begin{table}[h]\label{table}
\begin{center}
\begin{tabular}{|c||c|c|c|}
\hline
$(\mathcal{HT}^p_{q,\alpha}(\Bn))^*\sim$& $0<p<1$ &$p=1$ &$1<p<\infty$ \\
\hline
\hline
$0<q\leq 1$& $\mathcal{B}(\Bn)$ &$\mathcal{B}(\Bn)$ &$\mathcal{BT}^{p'}(\Bn)$ \\
\hline
$1<q<\infty$&$\mathcal{B}(\Bn)$ &$\mathcal{CT}_{q',\alpha}(\Bn)$ & $\mathcal{HT}^{p'}_{q',\alpha}(\Bn)$ \\
\hline
\end{tabular}
\caption{Duality results summarized}
\end{center}
\end{table}

There are also the natural "little-O" versions of the spaces $\mathcal{BT}^p(\Bn)$ and $\mathcal{CT}_{q,\alpha}(\Bn)$. These will be denoted by $\mathcal{BT}^{p,0}(\Bn)$ and $\mathcal{CT}^0_{q,\alpha}(\Bn)$, respectively. For $1<p<\infty$, we obtain the duality
$$(\mathcal{BT}^{p,0}(\Bn))^*\sim \mathcal{HT}^{p'}_{1,\alpha}(\Bn).$$
This our Theorem \ref{duality2}. For $1<q<\infty$, we obtain
$$\mathcal{CT}^0_{q,\alpha}(\Bn)\asymp \mathcal{HT}^1_{q',\alpha}(\Bn).$$
The case of $\mathcal{CT}^0_{q,\alpha}(\Bn)$ can be treated by using vanishing Carleson measures, once the main duality result is obtained. However, the definition of the space $\mathcal{BT}^{p,0}(\Bn)$ seems to be more involving, and we present the full proof of the duality in its own section.

This work can be considered as a unified treatment of duality results that contains also the classical Hardy, Hardy-BMOA, as well as the Bergman and Bergman-Bloch dualities. The end-point spaces $\mathcal{BT}^p(\Bn)$ and $\mathcal{CT}_{q,\alpha}(\Bn)$ do not seem to have appeared in the literature in this context and generality. In the author's opinion, they clarify the picture of the classical dualities in very natural way.

Having the boundedness of the Bergman projection opens possibilities for a rather straightforward study of many operators. These include the corresponding Toeplitz and Hankel operators, where the generating symbol can be assumed to be essentially bounded. This, for instance, makes it natural to study the spectral properties of Toeplitz operators acting on $\mathcal{HT}^p_{q,\alpha}(\Bn)$ in a manner similar to the fundamental results of U. Venugopalkrisna, L. A. Coburn and L. Boutet de Monvel, \cite{BdM, Cob, Ven}. Let us also mention the fairly recent joint paper with A. B\"ottcher \cite{BP}.

Aside from these topics, we will end up with some open questions. These will be discussed in the last section of the paper. The author would like to remark that there some result, most notably Lemmas \ref{seq}, \ref{converse} and \ref{embed2} that, although proven here, are most likely well-known. The author claims no originality of these results, but decided to add proofs due the difficulty of finding them in the literature, at least in the generality considered here.\\

\subsection{Notation and conventions}

We will throughout use the following conventions. Given a quasinormed space $X$ and $x \in X$, the quasinorm of $x$ on $X$ is denoted by $\|x\|_X$. If $T:X\to Y$ is a linear operator between two quasinormed spaces $X$ and $Y$, then its operator quasinorm is denoted by $\|T\|_{X\to Y}$. The space of continuous linear functionals $X\to \mathbb{C}$ -- or the dual of $X$ -- is denoted by $X^*$. We will also write $X\sim Y$, to express that the spaces are isomorphic.

We will be studying several function spaces with some $p$-integrability condition. Given $p \in [1,\infty]$, we will always denote by $p'$ its dual exponent: $p'=p/(p-1)$. We understand $1'=\infty$ and $\infty'=1$.

By $A \lesssim B$, we mean that there exists $C>0$ so that $A \leq CB$. The relation $A \gtrsim B$ is defined analogously. If both $A \lesssim B$ and $A \gtrsim B$, we write $A\asymp B$.\\

\subsection*{Acknowledments}
The author wishes to thank Jordi Pau for several inspiring discussions and showing him the power of tent spaces.\\

\section{Preliminaries}

\noindent Denote by $\Bn=\{z\in \mathbb{C}^n:|z|<1\}$ the open unit ball in $\Cn$, the Euclidean space of complex dimension $n$. For any two points $z=(z_ 1,\dots,z_ n)$ and $w=(w_ 1,\dots,w_ n)$ in $\Cn$ we write
\[
\langle z,w\rangle =z_ 1\bar{w}_ 1+\dots +z_ n \bar{w}_ n,
\]
 and
\[
 |z|=\sqrt{\langle z,z\rangle}=\sqrt{|z_ 1|^2+\dots +|z_ n|^2}.
\]

Let $\beta \in \mathbb{R}$ and $0<p<\infty$. The Lebesgue space $L^p_\beta(\Bn)$ consists of measurable $f$ on $\Bn$ with
$$\|f\|_{L^p_\beta(\Bn)}^p=\int_{\Bn} |f(z)|^p (1-|z|^2)^\beta dV_n(z) < \infty,$$
where $dV_n(z)$ is the $2n$-dimensional Lebesgue measure on $\Bn$, which is normalized so that $V_n(\Bn)=1$. The space $L^\infty(\Bn)$ consists of the essentially bounded measurable functions $f$, and the norm used will be
$$\|f\|_{L^\infty(\Bn)}=\operatorname{ess}\sup_{z \in \Bn}|f(z)|.$$

\subsection{Hardy, Bergman, BMOA and Bloch spaces}

For $0<p<\infty$, the Hardy space $H^p(\Bn)$ consists of those holomorphic functions $f:\Bn\to \mathbb{C}$ with
 \[ \|f\|_{H^p(\Bn)}^p=\sup_{0<r<1}\int_{\Sn} \!\! |f(r\zeta)|^p \,d\sigma(\zeta)<\infty,\]
where $d\sigma$ is the surface measure on the unit sphere $\Sn:=\partial \Bn$ normalized so that $\sigma(\Sn)=1$. By $H^\infty(\Bn)$ we mean the space of bounded holomorphic functions on $\Bn$.

A function $f$ in $H^p(\Bn)$ has radial limits $f(\zeta)=\lim_{r\to 1^{-}} f(r\zeta)$ for almost every $\zeta \in \Sn$; and $H^2(\Bn)$ becomes a Hilbert space when endowed with the inner product

$$\langle f,g\rangle_{-1}= \int_{\Sn} f(\zeta)\overline{g(\zeta)}d\sigma(\zeta).$$

It is well-known that if $1<p<\infty$, then $H^p(\Bn)^*\sim H^{p'}(\Bn)$. The dual of $H^1(\Bn)$ is known as the space $BMOA(\Bn)$; this fact will be discussed in more detail later on.

Let $\beta>-1$ and $0<p<\infty$. The Bergman space $A^p_\beta(\Bn)$ consists of those holomorphic functions $f$ in $\Bn$, for which
$$\|f\|_{A^p_\beta(\Bn)}^p=\int_{\Bn} |f(z)|^p (1-|z|^2)^\beta dV_n(z) < \infty.$$
The space $A^2_\beta(\Bn)$ is a Hilbert space under the inner product
$$\langle f,g\rangle_{\beta}=\int_{\Bn} f(z)\overline{g(z)}(1-|z|^2)^\beta dV_n(z).$$

The orthogonal projection $P_\beta:L^2_\beta(\Bn)\to A^2_\beta(\Bn)$ is called the Bergman projection, and it is given by the formula
$$P_\beta f(z)=c(n,\beta)\int_{\Bn}\frac{(1-|u|^2)^\beta}{(1-\langle z,u\rangle)^{1+n+\beta}} f(u)dV_n(u).$$
Here $$c(n,\beta)=\frac{\Gamma(n+\beta+1)}{n!\Gamma(\beta+1)}$$ is the constant that makes $(1-|z|^2)^\beta dV_n(z)$ a probability measure. It is well-known that $P_\beta:L^p_\beta(\Bn)\to A^p_\beta(\Bn)$ is bounded if and only if $1<p<\infty$. 

We will also need the Bloch space $\mathcal{B}(\Bn)$, which consists of those holomorphic functions $f$ on $\Bn$, for which 
$$\|f\|_{\mathcal{B}(\Bn)}=\sup_{z \in \Bn} (1-|z|^2)|\nabla f(z)|+|f(0)|<\infty,$$
where $$\nabla f=(\partial_{z_1}f, \partial_{z_2}f,\dots, \partial_{z_n}f)$$ is the complex gradient of $f$. It is known that $P_\beta(L^\infty(\Bn))=\mathcal{B}(\Bn)$. Moreover, if $1<p<\infty$, then $A^p_\beta(\Bn)^*\sim A^{p'}_\beta(\Bn)$, whereas the dual of $A^1_\beta(\Bn)$ can be identified as the Bloch space.

We refer to the books \cite{Rud} and \cite{ZhuBn} for the theory of Hardy, Bergman and Bloch spaces of the unit ball. For Hardy spaces of the unit disk, see \cite{duren1}.\\

\subsection{Fractional differential operators}

We will be using a family of fractional derivative and integral operators. Suppose that $s$ and $t$ are real parameters such that neither $n+s$ nor $n+s+t$ is a negative integer. Then the fractional differential operator $R^{s,t}$ is the unique linear operator, which is continuous on the space of holomorphic functions on the ball (equipped with the topology of uniform convergence on compact sets) satisfying
$$R^{s,t}\frac{1}{(1-\langle z,u\rangle)^{n+1+s}}=\frac{1}{(1-\langle z,u\rangle)^{n+1+s+t}}$$
In a similar way, we define the fractional integration operator $R_{s,t}$ by the relation
$$R_{s,t}\frac{1}{(1-\langle z,u\rangle)^{n+1+s+t}}=\frac{1}{(1-\langle z,u\rangle)^{n+1+s}}.$$
It clearly holds that 
$$R^{s,t}R_{s,t}=R_{s,t}R^{s,t}=Id.$$

These fractional operators are particularly useful in the following situation. Suppose that  the holomorphic $f$ has the integral representation
$$f(z)=\int_{\Bn}\frac{d\mu(u)}{(1-\langle z,u\rangle)^{n+1+s}}.$$
Then 
$$R^{s,t}f(z)=\int_{\Bn}\frac{d\mu(u)}{(1-\langle z,u\rangle)^{n+1+s+t}}.$$

If $s>-1$, $t>0$ and $f \in A^1_s(\Bn)$, then
\begin{equation}\label{FracDer}
R^{s,t}=c(n,s)\int_{\Bn}\frac{(1-|u|^2)^s}{(1-\langle z,u\rangle)^{1+n+s+t}}f(u)dV_n(u).
\end{equation}
In a similar fashion, if $t>0$, $s+t>-1$ and $f \in A^1_{s+t}(\Bn)$, we can write
\begin{equation}\label{FracInt}
R_{s,t}f(z)=c(n,s+t)\int_{\Bn}\frac{(1-|u|^2)^{s+t}}{(1-\langle z,u\rangle)^{1+n+s}}f(u)dV_n(u).
\end{equation}

The next lemma will be needed to carry out some calculations, when the parameters $s$ and $t$ do not match those of the kernel.

\begin{otherl}\label{change}
Suppose that neither $n+s$ nor $n+s+t$ is a negative integer, and that $N$ is a positive integer. Then, there exist one variable polynomials $\phi$ and $\psi$ of degree $N$ so that
\begin{equation}\label{pol1}
R^{s,t}\frac{1}{(1-\langle z,u\rangle)^{n+1+s+N}}=\frac{\phi(\langle z,u\rangle)}{(1-\langle z,u\rangle)^{n+1+s+N+t}}
\end{equation}
and
\begin{equation}\label{pol2}
R_{s,t}\frac{1}{(1-\langle z,u\rangle)^{n+1+s+N+t}}=\frac{\psi(\langle z,u\rangle)}{(1-\langle z,u\rangle)^{n+1+s+N}}.
\end{equation}
\end{otherl}

This lemma will most often be used in a form, where the polynomials are estimated from above by their $L^\infty(\Bn)$ norms.

We refer an interested reader to the monograph of R. Zhao and K. Zhu \cite{ZZ} and the textbook of Zhu \cite{ZhuBn} for detailed treatment of these fractional derivatives.\\

\subsection{Tent spaces}

For $\zeta \in \Sn$ and $\gamma>1$, recall that the admissible Kor\'anyi approach region $\Gamma_{\gamma}(\zeta)$ is defined as
\begin{displaymath}
\Gamma(\zeta)=\Gamma_{\gamma}(\zeta)=\left \{z\in \Bn: |1-\langle z,\zeta\rangle |<\frac{\gamma}{2} (1-|z|^2) \right \}.
\end{displaymath}

As a consequence of Lemma \ref{Gamma} below, it follows that for our purposes, the choice of $\gamma>1$ is not important. Therefore, as is customary, we will suppress it from the notation.

We set $I(0)=\Sn$. If $z\neq 0$, then $I(z)=\{\zeta \in \Sn: z\in \Gamma(\zeta)\}$. It can be shown that $\sigma(I(z))\asymp (1-|z|^2)^{n}$, and it follows from Fubini's theorem that, for a positive function $\varphi$, and a finite positive measure $\nu$, one has
\begin{equation}\label{EqG}
\int_{\Bn} \varphi(z)\,d\nu(z)\asymp \int_{\Sn} \left (\int_{\Gamma(\zeta)} \varphi(z) \frac{d\nu(z)}{(1-|z|^2)^{n}} \right )d\sigma(\zeta).
\end{equation}

This fact will be crucial for the analysis of tent spaces.

Let $0<p,q<\infty$ and $\alpha>-n-1$. The weighted tent space $T^p_{q,\alpha}(\Bn)$ consists of those measurable $f$ on $\Bn$ with
$$\|f\|_{T^p_{q,\alpha}(\Bn)}^p=\int_{\Sn}\left(\int_{\Gamma(\zeta)}|f(z)|^q(1-|z|^2)^\alpha dV_n(z)\right)^{p/q}d\sigma(\zeta).$$ In addition, $T^p_\infty(\Bn)$ consists of measurable functions $f$ with
$$\|f\|_{T^p_\infty(\Bn)}^p=\int_{\Sn}\left(\operatorname{ess} \sup_{z \in \Gamma(\zeta)} |f(z)|\right)^{p}d\sigma(\zeta)<\infty.$$

Note that this definition of $T^p_\infty(\Bn)$ differs from the original one in \cite{CMS}, but has since been used in the present context, see \cite{Lue1, Lue2, Pel}. In order to avoid some cumbersome statements, we understand that $T^p_{\infty,\alpha}(\Bn)=T^p_\infty(\Bn)$.

The quantity $\|\cdot\|_{T^p_{q,\alpha}(\Bn)}$ is obviously a norm if $p,q\geq 1$. Moreover, if either $p<1$ or $q<1$ (or both), the expression
$$\|f-g\|_{T^p_{q,\alpha}(\Bn)}^{\min(p,q)}$$
is a metric on $T^p_{q,\alpha}(\Bn)$. It is clear that if $p_1\leq p_2$ and $q_1\leq q_2$, we have
$$T^{p_2}_{q_2,\alpha}(\Bn)\subset T^{p_1}_{q_1,\alpha}(\Bn)$$
with continuous inclusion.

For non-zero $u \in \Bn$, we define $\zeta_u=u/|u|$ and set
$$Q(u)=\{ z\in \Bn: |1-\langle z,\zeta_u\rangle|<1-|u|^2\}.$$ We agree that $Q(0)=\Bn$.
The space $T^\infty_{q,\alpha}(\Bn)$ consists of measurable functions $f$ with
$$\|f\|_{T^\infty_{q,\alpha}(\Bn)}=\operatorname{ess}\sup_{\zeta \in \Sn} \left(\sup_{u \in \Gamma(\zeta)}\frac{1}{(1-|u|^2)^n}\int_{Q(u)}|f(z)|^q (1-|z|^2)^{n+\alpha}dV_n(z)\right)^{1/q}.$$
By comparing with discussion in Section 5.2 of \cite{ZhuBn}, we notice that $f \in T^\infty_{q,\alpha}(\Bn)$ if and only if $(1-|\cdot|^2)^{n+\alpha}|f|^q$ is a Carleson measure on $\Bn$. Motivated by this fact, we recall a characterization for Carleson measures, see for instance Theorem 45 of \cite{ZZ}.

\begin{otherl}\label{CM}
Let $\mu$ be a positive Borel measure. Then $\mu$ is a Carleson measure if and only if
\begin{equation}\label{CMsup}
\sup_{a \in \Bn} \int_{\Bn}\frac{(1-|a|^2)^T}{|1-\langle z,a\rangle|^{n+T}}d\mu(z)<\infty
\end{equation}
for some, or equivalently all, $T>0$.
\end{otherl}

It follows that if $\|\mu\|_{CM_T(\Bn)}$ denotes the supremum in \eqref{CMsup}, and $$d\mu(z)=(1-|z|^2)^{n+\alpha}|f(z)|^q dV_n(z),$$ then
$$\|\mu\|_{CM_T(\Bn)}^{1/q}\asymp \|f\|_{T^\infty_{q,\alpha}(\Bn)}.$$

The following result is the standard duality theorem for measurable tent spaces in a slightly less general form that will suit our purposes. See \cite{CMS, Lue1}, and also \cite{Pel}.

\begin{otherth}\label{dual}
Let $\alpha>-n-1$, $1\leq p<\infty$, $1\leq q<\infty$ and $p+q\neq 2$. Then, under the pairing
$$\langle f,g\rangle_{n+\alpha}=\int_{\Bn}f(z)\overline{g(z)}(1-|z|^2)^{n+\alpha} dV_n(z)$$
the dual of $T^p_{q,\alpha}(\Bn)$ is $T^{p'}_{q',\alpha}(\Bn)$ with equivalent norms.
\end{otherth}

Our goal is to establish a satisfactory holomorphic version of this theorem.

The integral estimate below is one of the key parts of our argument. This result can be found in \cite{Ars} and \cite{Jev}. The converse inequality is trivially valid for every $\lambda$.

\begin{otherl}\label{Gamma}
Let $0<s<\infty$ and $\lambda>n\max(1,1/s)$. If $\mu$ is a positive measure, then
$$\int_{\Sn}\left[\int_{\Bn}\left(\frac{1-|z|^2}{|1-\langle z,\zeta\rangle|}\right)^\lambda d\mu(z)\right]^s d\sigma(\zeta)\lesssim \int_{\Sn}\mu(\Gamma(\zeta))^s d\sigma(\zeta).$$
\end{otherl}

\subsection{Forelli-Rudin estimates}

We recall some important integral estimates involving Bergman type kernels on balls. The first one is the classical Forelli-Rudin estimate (see \cite{FR} or \cite[Theorem 1.12]{ZhuBn} for example). The sharp constants for the estimate are also known, see \cite{Liu}.
\begin{otherl}\label{IctBn}
Let $t>-1$ and $s>0$. Then,
\[ \int_{\Bn} \frac{(1-|u|^2)^t\,dV_n(u)}{|1-\langle z,u\rangle |^{n+1+t+s}}\lesssim (1-|z|^2)^{-s}\]
for all $z\in \Bn$.
\end{otherl}

The following, more general version of Lemma \ref{IctBn} is, along with Lemma \ref{Gamma}, perhaps the most used estimate in this paper. The proof can found in \cite{OF}. By, for instance, Vitali's convergence theorem, we may sometimes assume that the points are in $\overline{\Bn}$.

\begin{otherl}\label{FRgeneral}
Let $s>-1$, $s+n+1>r,t>0$, and $r+t-s>n+1$. For $u,z \in \overline{\Bn}$, one has
\begin{equation}\label{FR1}
\int_{\Bn} \frac{(1-|w|^2)^s dV_n(w)}{|1-\langle u,w\rangle|^r |1-\langle z,w\rangle|^t} \lesssim \frac{1}{|1-\langle z,u\rangle|^{r+t-s-n-1}}.
\end{equation}
If $s>-1$, $r,t>0$, $t>s+n+1>r$ and $r+t>s+n+1$, then for $z \in \Bn$ and $u \in \overline{\Bn}$,
\begin{equation}\label{FR2}\int_{\Bn} \frac{(1-|w|^2)^s dV_n(w)}{|1-\langle u,w\rangle|^r |1-\langle z,w\rangle|^t} \lesssim \frac{1}{|1-\langle z,u\rangle|^{r}(1-|z|^2)^{t-s-n-1}}.
\end{equation}

\end{otherl}

We will mostly use the formula \eqref{FR1}, and will explicitly state when \eqref{FR2} is used.\\

\section{Holomorphic tent spaces $\mathcal{HT}^p_{q,\alpha}(\Bn)$ and $\mathcal{BT}^p(\Bn)$}

\noindent In this section we will define the holomorphic tent spaces $\mathcal{HT}^p_{q,\alpha}(\Bn)$ and $\mathcal{BT}^p(\Bn)$, and characterize them in terms of the fractional derivatives.

\begin{definition}
Let $0<p<\infty$, $0<q\leq\infty$ and $\alpha>-n-1$. The Hardy type tent space $\mathcal{HT}^p_{q,\alpha}(\Bn)$ consists of holomorphic functions on $\Bn$ that also belong to $T^p_{q,\alpha}(\Bn)$. The quasinorm of $f \in \mathcal{HT}^p_{q,\alpha}(\Bn)$ will be given by $\|f\|_{\mathcal{HT}^p_{q,\alpha}(\Bn)}=\|f\|_{T^p_{q,\alpha}(\Bn)}$.
\end{definition}

It is not difficult to see that every $\mathcal{HT}^p_{q,\alpha}(\Bn)$ belongs to some $A^1_\beta(\Bn)$. In particular, the standard reproducing formulas and fractional differential operators can be used, albeit one must be somewhat careful.

We next present the area function description of the Hardy spaces, which is the motivation for our terminology "Hardy type tent space". It is usually expressed in terms of some other differential operators, see \cite{AB, FS}; for a proof of the present variant, see \cite{PP}.   This name is also chosen to distinguish from the "Bergman type tent spaces" studied in \cite{PR2015}.

\begin{otherth}\label{area}
Let $0<p<\infty$, $t>0$ and neither $n+s$, nor $n+s+t$ is a negative integer. The holomorphic function $f$ belongs to the Hardy space $H^p(\Bn)$ if and only if $R^{s,t}f \in \mathcal{HT}^p_{2,2t-1-n}(\Bn)$. Moreover,  
$$\|f\|_{H^p(\Bn)}\asymp \|R^{s,t}f\|_{\mathcal{HT}^p_{2,2t-1-n}(\Bn)}.$$
\end{otherth}

Note that one of the benefits of using $R^{s,t}$ is that we automatically get a quasinorm, without the need to introduce a point-evaluation at zero.

By the non-tangential maximal function characterization of the Hardy space, $\mathcal{HT}^p_{\infty,\alpha}(\Bn)=H^p(\Bn)$, see for instance \cite{Rud}. By \eqref{EqG}, we also have $\mathcal{HT}^p_{p,\alpha}(\Bn)=A^p_{n+\alpha}(\Bn)$.

For the proof of theorems \ref{HTchar} and \ref{BTchar}, it will be in order to explain our choice of parameters beforehand. The parameters $N$ and $M$ denote some large positive integers, and their purpose is to make the use of lemmas \ref{change}, \ref{Gamma} and \ref{FRgeneral} legal. We also use $\lambda>0$, which is usually assumed to be large compared to $N$ and $M$, in order to complete the use of Lemma \ref{FRgeneral}. Without specifying the exact value of these integers, we note that by the end of the proof, it will be clear that everything works since $N$ and $M$ can be made large. Bearing this in mind will make the presentation easier to understand.

We will next describe the spaces $\mathcal{HT}^p_{q,\alpha}(\Bn)$ in terms of the fractional derivatives $R^{s,t}$.

\begin{theorem}\label{HTchar}
Let $0<p,q<\infty$, $\alpha>-n-1$ and $s,t$ be real numbers. Assume further that neither $s+n$ nor $s+n+t$ is a negative integer and $qt+n+1+\alpha>0$. Then the following are equivalent.
\begin{itemize}
\item[(a)] $f \in \mathcal{HT}^p_{q,\alpha}(\Bn)$;
\item[(b)] $R^{s,t}f \in \mathcal{HT}^p_{q,\alpha+qt}(\Bn)$. 
\end{itemize}
Moreover, the corresponding quasinorms are equivalent.
\end{theorem}

\begin{proof}
Assume first that $q\geq 1$ (recall that $1'=\infty$). For a sufficiently large positive integer $N$, we have by H\"older's inequality and Lemmas \ref{change} and \ref{IctBn},
\begin{align*}
|R^{s,t}f(z)|^q &\lesssim \left(\int_{\Bn} \frac{(1-|u|^2)^{s+N}}{|1-\langle z,u\rangle|^{1+n+s+N+t}}|f(u)|dV_n(u)\right)^q\\
&\lesssim (1-|z|^2)^{-(qt+n+1+\alpha)/q'}\int_{\Bn}\frac{(1-|u|^2)^{s+N}}{|1-\langle z,u\rangle|^{1+n+s+N+t-(n+1+\alpha)/q'}}|f(u)|^qdV_n(u).
\end{align*}
If $z \in \Gamma(\zeta)$, then $(1-|z|^2)\asymp |1-\langle z,\zeta\rangle|$, so for a sufficiently large $\lambda>0$, we are allowed to use Lemma \ref{FRgeneral} together with Fubini's theorem to get
\begin{align*}
&\int_{\Gamma(\zeta)}|R^{s,t}f(z)|^q(1-|z|^2)^{qt+\alpha}dV_n(z)\\
\lesssim &\int_{\Gamma(\zeta)}(1-|z|^2)^{t+\alpha-(n+1+\alpha)/q'}\int_{\Bn}\frac{(1-|u|^2)^{s+N}}{|1-\langle z,u\rangle|^{1+n+s+N+t-(n+1+\alpha)/q'}}|f(u)|^qdV_n(u)\\
\lesssim &\int_{\Bn}\frac{(1-|z|^2)^{\lambda+t+\alpha-(n+1+\alpha)/q'}}{|1-\langle z,\zeta\rangle|^{\lambda}}\int_{\Bn}\frac{(1-|u|^2)^{s+N}}{|1-\langle z,u\rangle|^{1+n+s+N+t-(n+1+\alpha)/q'}}|f(u)|^qdV_n(u)\\
\lesssim&\int_{\Bn}\left(\frac{(1-|u|^2)}{|1-\langle u,\zeta\rangle|}\right)^{s+N-\alpha}|f(u)|^q (1-|u|^2)^\alpha dV_n(u).
\end{align*}
Now, since $N$ can be arbitrarily large, we are able to use Lemma \ref{Gamma} to obtain
\begin{align*}
&\int_{\Sn}\left(\int_{\Gamma(\zeta)}|R^{s,t}f(z)|^q(1-|z|^2)^{qt+\alpha}dV_n(z)\right)^{p/q}d\sigma(\zeta)\\
\lesssim &\int_{\Sn}\left(\int_{\Bn}\left(\frac{(1-|u|^2)}{|1-\langle u,\zeta\rangle|}\right)^{s+N-\alpha}|f(u)|^q (1-|u|^2)^\alpha dV_n(u)\right)^{p/q}d\sigma(\zeta)\\
\lesssim &\int_{\Sn}\left(\int_{\Gamma(\zeta)}|f(u)|^q(1-|u|^2)^{\alpha}dV_n(u)\right)^{p/q}d\sigma(\zeta).
\end{align*}
This shows that (a) implies (b).\\

Next, let us assume that (b) holds. We use the identity $R_{s,t}R^{s,t}f=f$ and Lemma \ref{change} to write for a sufficiently large positive integer $N$ and $\varepsilon>0$ small enough so that $\alpha-\varepsilon q>-n-1$

\begin{align*}
|f(z)|^q&\lesssim \left(\int_{\Bn}\frac{(1-|u|^2)^{s+N+t}}{|1-\langle z,u\rangle|^{1+n+s+N}}|R^{s,t}f(u)|dV_n(u)\right)^q\\
&\lesssim (1-|z|^2)^{-\varepsilon q}\int_{\Bn}\frac{(1-|u|^2)^{s+N+tq+\varepsilon q}}{|1-\langle z,u\rangle|^{1+n+s+N}}|R^{s,t}f(u)|^qdV_n(u).
\end{align*}

As before, we will use Lemma \ref{FRgeneral} and Fubini's theorem, this time omitting some of the details, to arrive at
\begin{align*}
&\int_{\Gamma(\zeta)}|f(z)|^q(1-|z|^2)^{\alpha}dV_n(z)\\
\lesssim &\int_{\Bn}\left(\frac{(1-|u|^2)}{|1-\langle u,\zeta\rangle|}\right)^{s+N-\alpha+\varepsilon q}|R^{s,t}f(u)|^q (1-|u|^2)^{qt+\alpha} dV_n(u).
\end{align*}
Because $N$ was allowed to be very large, an application of Lemma \ref{Gamma} yields (a).\\

Let us now turn our attention to the case $q<1$. By Lemma 3 of \cite{ZhuSmall}, the following inequality holds for any holomorphic $F:\Bn\to \mathbb{C}$:
$$\int_{\Bn}|F(z)|(1-|z|^2)^{\eta+(1/q-1)(n+1+\eta)}dV_n(z)\lesssim \left(\int_{\Bn}|F(z)|^q (1-|z|^2)^\eta dV_n(z)\right)^{1/q},$$
when $0<q\leq 1$ and $\eta>-1$. Hence for a large positive integer $N$, if
$$s+N=\eta+(1/q-1)(n+1+\eta),$$ then 
$$\eta=q(s+N)+(q-1)(n+1)>-1,$$
and we have by Lemma \ref{change}
\begin{align*}
|R^{s,t}f(z)|&\lesssim \int_{\Bn} \frac{(1-|u|^2)^{s+N}}{|1-\langle z,u\rangle|^{1+n+s+N+t}}|f(u)|dV_n(u)\\
&\lesssim \left(\int_{\Bn}\frac{(1-|u|^2)^{q(s+N)+(q-1)(n+1)}}{|1-\langle z,u\rangle|^{q(1+n+s+N+t)}}|f(u)|^qdV_n(u)\right)^{1/q}.
\end{align*}

Now, for large positive $\lambda$, we will use Lemma \ref{FRgeneral} and Fubini's theorem to estimate
\begin{align*}
&\int_{\Gamma(\zeta)}|R^{s,t}f(z)|^q (1-|z|^2)^{qt+\alpha}dV_n(z)\\
\lesssim &\int_{\Gamma(\zeta)} \left(\int_{\Bn}\frac{(1-|u|^2)^{q(s+N)+(q-1)(n+1)}}{|1-\langle z,u\rangle|^{q(1+n+s+N+t)}}|f(u)|^q dV_n(u)\right)(1-|z|^2)^{qt+\alpha}dV_n(z)\\
\lesssim &\int_{\Bn} \frac{(1-|z|^2)^{\lambda+qt+\alpha}}{|1-\langle z,\zeta\rangle|^\lambda} \left(\int_{\Bn}\frac{(1-|u|^2)^{q(s+N)+(q-1)(n+1)}}{|1-\langle z,u\rangle|^{q(1+n+s+N+t)}}|f(u)|^q dV_n(u)\right)dV_n(z)\\
\lesssim &\int_{\Bn} \left(\frac{(1-|u|^2)}{|1-\langle u,\zeta\rangle|}\right)^{(q-1)(n+1)+q(s+N)-\alpha}|f(u)|^q (1-|u|^2)^\alpha dV_n(u).
\end{align*}
By allowing $N$ to be very large, we can use Lemma \ref{Gamma} to conclude that (a) implies (b).\\

The remaining case can be solved by techniques exhibited in the earlier cases. The proof is left for an interested reader.

\end{proof}

To understand the dual space of $\mathcal{HT}^p_{1,\alpha}(\Bn)$, we will introduce the following Hardy-Bloch type spaces $\mathcal{BT}^p(\Bn)$. 

\begin{definition}
Let, $1<p<\infty$, $t>0$, and assume that neither $n+s$ nor $n+s+t$ is a negative integer. A holomorphic $f$ belongs to the space $\mathcal{BT}^p(\Bn)$, if 
\begin{equation}\label{BT}
(1-|\cdot|^2)^t R^{s,t}f \in T^p_\infty(\Bn).
\end{equation}
Moreover, the expression
$$\|(1-|\cdot|^2)^tR^{s,t}f\|_{T^p_\infty(\Bn)}$$ is a norm on $\mathcal{BT}^p(\Bn)$.
\end{definition}

We remark that the following inclusion holds
$$H^p(\Bn)\cup \mathcal{B}(\Bn) \subset \mathcal{BT}^p(\Bn)$$ for every $\alpha>-n-1$.

It might seem that the space $\mathcal{BT}^p(\Bn)$ depends on the choice of $s$ and $t$ is some crucial way. We now prove that this is not the case.

\begin{theorem}\label{BTchar}
Let $1<p<\infty$, $t_1,t_2>0$ and $s_1,s_2$ be a real numbers. Assume further that none of the numbers $s_1+n$, $s_2+n$, $s_1+n+t_1$ and $s_2+n+t_2$ is a negative integer. Then the following are equivalent.
\begin{itemize}
\item[(a)] $(1-|\cdot|^2)^{t_1} R^{s_1,t_1} f \in T^p_\infty(\Bn)$;
\item[(b)] $(1-|\cdot|^2)^{t_2} R^{s_2,t_2} f \in T^p_\infty(\Bn)$. 
\end{itemize}
Moreover, the corresponding norms are equivalent.
\end{theorem}

\begin{proof}
By symmetry, it will be enough to show that (a) implies (b). Since the fractional derivatives commute with each other, we have
$$R^{s_2,t_2}f=R^{s_2,t_2}R_{s_1,t_1}R^{s_1,t_1}f=R_{s_1,t_1}R^{s_2,t_2}R^{s_1,t_1}f.$$
it follows that if $N$ is a large positive integer, then by Lemma \ref{change}
$$|R^{s_2,t_2}f(z)|\lesssim \int_{\Bn}\frac{(1-|u|^2)^{s_1+N+t_1}}{|1-\langle z,u\rangle|^{1+n+s_1+N}}|R^{s_2,t_2}R^{s_1,t_1}f(u)|dV_n(u).$$
Further, if $M$ is a large positive integer, we can use Lemma \ref{change} again to estimate
$$|R^{s_2,t_2}R^{s_1,t_1}f(u)|\lesssim \int_{\Bn} \frac{(1-|w|^2)^{s_2+M}}{|1-\langle u,w\rangle|^{1+n+s_2+M+t_2}}|R^{s_1,t_1}f(w)|dV_n(w).$$
Therefore, by using Fubini's theorem and Lemma \ref{FRgeneral}, it holds that
\begin{align*}
&|R^{s_2,t_2}f(z)|\\
\lesssim& \int_{\Bn}\frac{(1-|u|^2)^{s_1+N+t_1}}{|1-\langle z,u\rangle|^{1+n+s_1+N}}\left(\int_{\Bn} \frac{(1-|w|^2)^{s_2+M}}{|1-\langle u,w\rangle|^{1+n+s_2+M+t_2}}|R^{s_1,t_1}f(w)|dV_n(w)\right)dV_n(u)\\
\lesssim &\int_{\Bn}\frac{(1-|w|^2)^{s_2+M-t_1}}{|1-\langle z,w\rangle|^{1+n+s_2+M+t_2-t_1}}(1-|w|^2)^{t_1} |R^{s_1,t_1}f(w)|dV_n(w).
\end{align*}
We note that under the pairing $\langle \cdot,\cdot\rangle_{s_2+M-t_1}$, the operator $S^*:T^p_\infty(\Bn)\to T^p_\infty(\Bn)$
$$S^*g(z)= (1-|z|^2)^{t_2}\int_{\Bn}\frac{(1-|u|^2)^{s_2+M-t_1}}{|1-\langle z,u\rangle|^{1+n+s_2+M+t_2-t_1}}g(u)dV_n(u)$$
if the Banach space adjoint of
$$Sg(z)=\int_{\Bn}\frac{(1-|u|^2)^{s_2+M-t_1+t_2}}{|1-\langle z,u\rangle|^{1+n+s_2+M+t_2-t_1}}g(u)dV_n(u)$$
acting on $T^{p'}_{1,s_2+M-t_1-n}(\Bn)$. Now, for large positive $\lambda$, we have
\begin{align*}
&\int_{\Gamma(\zeta)}|Sg(z)|(1-|z|^2)^{s_2+M-t_1-n}dV_n(z)\\
\lesssim &\int_{\Gamma(\zeta)}(1-|z|^2)^{s_2+M-t_1-n}\int_{\Bn}\frac{(1-|u|^2)^{s_2+M-t_1+t_2}}{|1-\langle z,u\rangle|^{1+n+s_2+M+t_2-t_1}}|g(u)|dV_n(u)\\
\lesssim &\int_{\Bn} \frac{(1-|z|^2)^{\lambda+s_2+M-t_1-n}}{|1-\langle z,\zeta\rangle|^\lambda}\left(\int_{\Bn}\frac{(1-|u|^2)^{s_2+M-t_1+t_2}}{|1-\langle z,u\rangle|^{1+n+s_2+M+t_2-t_1}}|g(u)|dV_n(u)\right)dV_n(z)\\
\lesssim &\int_{\Bn}\left(\frac{(1-|u|^2)}{|1-\langle u,\zeta\rangle|}\right)^{t_2+n}|g(u)|(1-|u|^2)^{s_2+M-t_1-n}dV_n(z).
\end{align*}
Since $p'>1$, we can use Lemma \ref{Gamma} to conclude that $S$ and its adjoint $S^*$ are bounded. Returning to the estimate
$$|R^{s_2,t_2}f(z)|\lesssim \int_{\Bn}\frac{(1-|w|^2)^{s_2+M-t_1}}{|1-\langle z,w\rangle|^{1+n+s_2+M+t_2-t_1}}(1-|w|^2)^{t_1} |R^{s_1,t_1}f(w)|dV_n(w),$$
we see that multiplying both sides by $(1-|z|^2)^{t_2}$ and noting that then the right-hand-side depicts a function in $T^p_\infty(\Bn)$, we arrive at the desired conclusion. The proof is complete.
\end{proof}

It is perhaps also noteworthy that the space $\mathcal{BT}^p(\Bn)$ lies strictly between the Hardy space $H^p(\Bn)$ and every weighted Bergman space $A^p_\beta(\Bn)$, where $\beta>-1$. This could be compared with the small weighted Bergman spaces, studied in the monograph \cite{PRMEM}. \\

\section{Duality: the case $1<p<\infty$, $1\leq q<\infty$}

\noindent This section is devoted to the proof of our first main duality result. Since the duality of the measurable tent spaces $T^p_{q,\alpha}(\Bn)$ and $T^{p'}_{q',\alpha}(\Bn)$ is realized in terms of the pairing
$$\langle f,g\rangle_{n+\alpha}=\int_{\Bn}f(z)\overline{g(z)}(1-|z|^2)^{n+\alpha}dV_n(z),$$
a natural course of action will be to study the boundedness of $P_{n+\alpha}:T^p_{q,\alpha}(\Bn)\to T^p_{q,\alpha}(\Bn)$. We will first show that this operator is not bounded when $q=1$ or $q=\infty$.

\begin{proposition}\label{notbdd}
Let $0<p<\infty$ and $\alpha>-n-1$. Then $P_{n+\alpha}:T^p_\infty(\Bn)\to T^p_\infty(\Bn)$ is not bounded. If $1<p<\infty$, then $P_{n+\alpha}:T^p_{1,\alpha}(\Bn)\to T^p_{1,\alpha}(\Bn)$ is not bounded.
\end{proposition}

\begin{proof}
We will show that $P_{n+\alpha}:T^p_\infty(\Bn)\to T^p_\infty(\Bn)$ is not bounded, when $0<p<\infty$. When $1<p<\infty$, the other case follows from duality. 

It is clear that $L^\infty(\Bn)\subset T^p_{\infty,\alpha}(\Bn)$. Therefore, as it is well-known that $P_{n,\alpha}(L^\infty(\Bn))=\mathcal{B}(\Bn)$, we have
$$\mathcal{B}(\Bn)\subset P_{n,\alpha}(T^p_\infty(\Bn)).$$
Now, since $\mathcal{HT}^p_\infty(\Bn)=H^p(\Bn)$ and $H^p(\Bn)$ does not contain the Bloch space $\mathcal{B}(\Bn)$, we are done.
\end{proof}

Next, we will show that the projection is indeed bounded if $1<p,q<\infty$. In fact, we will show that the maximal Bergman projection

$$f\mapsto c(n,n+\alpha)\int_{\Bn}\frac{(1-|u|^2)^{n+\alpha}}{|1-\langle z,u\rangle|^{1+2n+\alpha}}f(u)dV_n(u)$$ is bounded.

\begin{proposition}\label{projection}
Let $1<p,q<\infty$ and $\alpha>-n-1$. Then, $P_{n+\alpha}:T^p_{q,\alpha}(\Bn)\to \mathcal{HT}^p_{q,\alpha}(\Bn)$ is bounded and onto. Moreover, if $1< p<\infty$, then $P_{n+\alpha}:T^p_\infty(\Bn)\to \mathcal{BT}^p(\Bn)$ is bounded and onto.
\end{proposition}

\begin{proof}
We will first deal with the case $1<q<\infty$. Let us take $f \in T^p_{q,\alpha}(\Bn)$. Since, $1<q<\infty$, we have by H\"older's inequality and Lemma \ref{IctBn}
\begin{align*}
|P_{n+\alpha}f(z)|^q&\lesssim \left(\int_ {\Bn} \frac{(1-|u|^2)^{n+\alpha-\beta q'}}{|1-\langle u,z\rangle|^{1+2n+\alpha}}dV_n(u)\right)^{q'/q}\int_{\Bn} \frac{(1-|u|^2)^{n+\alpha+\beta q}}{|1-\langle u,z\rangle|^{1+2n+\alpha}}|f(u)|^q dV_n(u)\\
&\lesssim (1-|z|^2)^{-q\beta}\int_{\Bn} \frac{(1-|u|^2)^{n+\alpha+\beta q}}{|1-\langle u,z\rangle|^{1+2n+\alpha}}|f(u)|^q dV_n(u),
\end{align*}
whenever $\beta$ is positive and
\begin{equation}\label{cond1}
\beta q'< n+1+\alpha.
\end{equation}
Note that if $z \in \Gamma(\zeta)$ and $\zeta \in \Sn$, then $1-|z|^2\asymp |1-\langle z,\zeta\rangle|$. Therefore, for any $\lambda>0$, we have
\begin{align*}
&\int_{\Sn} \left( \int_{\Gamma(\zeta)} |P_{n+\alpha}f(z)|^q (1-|z|^2)^\alpha dV_n(z)\right)^{p/q}d\sigma(\zeta)\\
\lesssim&\int_{\Sn} \left( \int_{\Gamma(\zeta)} \left(\frac{(1-|z|^2)}{|1-\langle z,\zeta\rangle|}\right)^{\lambda}|P_{n+\alpha}f(z)|^q (1-|z|^2)^\alpha dV_n(z)\right)^{p/q}d\sigma(\zeta)\\
\lesssim&\int_{\Sn} \left( \int_{\Bn} \left(\frac{(1-|z|^2)}{|1-\langle z,\zeta\rangle|}\right)^{\lambda}|P_{n+\alpha}f(z)|^q (1-|z|^2)^\alpha dV_n(z)\right)^{p/q}d\sigma(\zeta)\\
\lesssim& \int_{\Sn} \left( \int_{\Bn} \frac{(1-|z|^2)^{\lambda+\alpha-q\beta}}{|1-\langle z,\zeta\rangle|^\lambda}\left(\int_{\Bn} \frac{(1-|u|^2)^{n+\alpha+\beta q}}{|1-\langle u,z\rangle|^{1+2n+\alpha}}|f(u)|^q dV_n(u)\right) dV_n(z) \right)^{p/q}d\sigma(\zeta).
\end{align*}
Note that this is the (almost) trivial counterpart of Lemma \ref{Gamma}. 

Now, looking at the assumptions on the parameters of Lemma \ref{FRgeneral}, we see that since $\alpha>-n-1$, then by making $\lambda$ large and requiring that
\begin{equation}\label{cond2}
\beta q < n+1+\alpha,
\end{equation}
we are in position to estimate
\begin{align*}
&\int_{\Sn} \left( \int_{\Bn} \frac{(1-|z|^2)^{\lambda+\alpha-q\beta}}{|1-\langle z,\zeta\rangle|^\lambda}\left(\int_{\Bn} \frac{(1-|u|^2)^{n+\alpha+\beta q}}{|1-\langle u,z\rangle|^{1+2n+\alpha}}|f(u)|^q dV_n(u)\right) dV_n(z) \right)^{p/q}d\sigma(\zeta)\\
\lesssim&\int_{\Sn} \left( \int_{\Bn} \frac{(1-|u|^2)^{n+\beta q}}{|1-\langle \zeta,u\rangle|^{n+\beta q}} (1-|u|^2)^\alpha |f(u)|^q dV_n(u) \right)^{p/q} d\sigma(\zeta),
\end{align*}
where we have used Fubini's theorem to change the order of integration. 

Next, we want to use Lemma \ref{Gamma} for the measure
$$d\mu(z)=(1-|u|^2)^\alpha |f(u)|^q dV_n(u),$$
with $s=p/q$. In order to be able to do this, we must make sure that $$n+\beta q>n\max(1,q/p),$$ where the only problem might occur when $q>p$, and the real requirement would therefore be
\begin{equation}\label{cond3}
n+\beta q > nq/p.
\end{equation}
To proceed with the proof, we now assume that $p\geq q$. Since $\beta>0$, we can continue with Lemma \ref{Gamma}, to finally arrive at
\begin{equation*}
\|P_{n+\alpha}f\|_{\mathcal{HT}^p_{q,\alpha}(\Bn)}^p \lesssim \int_{\Sn} \left( \int_{\Gamma(\zeta)}  |f(u)|^q (1-|u|^2)^\alpha dV_n(u) \right)^{p/q} d\sigma(\zeta)=\|f\|_{T^p_{q,\alpha}(\Bn)}^p.
\end{equation*}
Note that the conditions \eqref{cond1} and \eqref{cond2} can be achieved by a small enough positive $\beta$, say $\beta=1/qq'$. Since $\beta$ is positive, the condition $n+\beta q>n\max(1,q/p)$ is trivially met.

So far we have shown that $P_{n+\alpha}:T^p_{q,\alpha}(\Bn)\to T^p_{q,\alpha}(\Bn)$ is bounded, whenever $p\geq q$. Next, recall that under the dual pairing
$$\langle f,g\rangle_{n+\alpha}=\int_{\Bn} f(z)\overline{g(z)}(1-|z|^2)^{n+\alpha}dV_n(z)$$
the operator $P_{n+\alpha}$ is self-adjoint. Therefore, $P_{n+\alpha}:T^p_{q,\alpha}(\Bn)\to T^p_{q,\alpha}(\Bn)$ is bounded, if and only if its adjoint operator $P_{n+\alpha}:T^{p'}_{q',\alpha}(\Bn)\to T^{p'}_{q',\alpha}(\Bn)$ is bounded. If now $q>p$, then $q'\leq p'$, and the boundedness of $P_{n+\alpha}:T^p_{q,\alpha}(\Bn)\to T^p_{q,\alpha}(\Bn)$ follows by duality.

We finally remark that since $\mathcal{HT}^p_{q,\alpha}(\Bn)\subset T^p_{q,\alpha}(\Bn)$, the projective operator $$P_{n+\alpha}:T^p_{q,\alpha}(\Bn)\to \mathcal{HT}^p_{q,\alpha}(\Bn)$$ is clearly onto. \\

Let us now turn our attention to the second case, so suppose that $q=\infty$. By Theorem \ref{BTchar}, it will be enough to show that if $f \in T^p_\infty(\Bn)$, then
$$\|(1-|\cdot|^2)^t R^{n+\alpha,t}P_{n+\alpha}f\|_{T^p_{\infty}(\Bn)}\lesssim \|f\|_{T^p_\infty(\Bn)}.$$

We note that the map
$$f\mapsto (1-|\cdot|^2)^t R^{n+\alpha,t}P_{n+\alpha}f$$ acting on $T^{p}_\infty(\Bn)$ is the adjoint in $\langle \cdot,\cdot \rangle_{n+\alpha}$ of 
$$f\mapsto P_{n+\alpha+t}f$$ acting on $T^{p'}_{1,\alpha}(\Bn)$. So, by duality, in order to prove the boundedness claim, it will be enough to show that
$$P_{n+\alpha+t}:T^{p'}_{1,\alpha}(\Bn)\to T^{p'}_{1,\alpha}(\Bn)$$
is bounded.

Let $f \in T^{p'}_{1,\alpha}(\Bn)$. As before, we estimate by using Lemma \ref{Gamma}, Lemma \ref{FRgeneral} and Fubini's theorem for a sufficiently large $\lambda>0$
\begin{align*}
\|P_{n+\alpha+t}f\|_{T^{p'}_{1,\alpha}(\Bn)}^{p'}&\lesssim \int_{\Sn} \left(\int_{\Gamma(\zeta)}(1-|z|^2)^\alpha \int_{\Bn} \frac{(1-|u|^2)^{n+\alpha+t}}{|1-\langle u,z\rangle|^{1+2n+\alpha+t}}|f(u)|dV_n(u)\right)^{p'} d\sigma(\zeta)\\
&\lesssim \int_{\Sn} \left(\int_{\Bn} \frac{(1-|z|^2)^{\lambda+\alpha}}{|1-\langle z,\zeta\rangle|^\lambda} \int_{\Bn} \frac{(1-|u|^2)^{n+\alpha+t}}{|1-\langle u,z\rangle|^{1+2n+\alpha+t}}|f(u)|dV_n(u)\right)^{p'} d\sigma(\zeta)\\
&\lesssim \int_{\Sn} \left(\int_{\Bn} \frac{(1-|u|^2)^{n+t}}{|1-\langle u,\zeta\rangle|^{n+t}}(1-|u|^2)^\alpha |f(u)|dV_n(u)\right)^{p'} d\sigma(\zeta)\\
&\lesssim \int_{\Sn} \left(\int_{\Gamma(\zeta)}(1-|u|^2)^\alpha |f(u)|dV_n(u)\right)^{p'} d\sigma(\zeta)\\
&=\|f\|_{T^{p'}_{1,\alpha}(\Bn)}.
\end{align*}
Note that since $p'> 1$, the requirements of Lemma \ref{Gamma} are easily satisfied. In fact this proves that
$$f \mapsto (1-|\cdot|^2)^t \int_{\Bn}\frac{(1-|u|^2)^{n+\alpha}}{|1-\langle \cdot,u\rangle|^{1+2n+\alpha+t}}dV_n(u)$$ is bounded $T^p_\infty(\Bn)\to T^p_\infty(\Bn)$ for every $t>0$.

We have thus shown that for any $1<p<\infty$, the operator $P_{n+\alpha}:T^p_\infty(\Bn)\to \mathcal{BT}^p(\Bn)$ is bounded. Finally, we will show that is also onto. This is not as trivial as in the previous cases, since $\mathcal{BT}^p(\Bn)$ is not a subspace of $T^p_\infty(\Bn)$, as is evident from Proposition \ref{notbdd}. Let us therefore take $g \in \mathcal{BT}^p(\Bn)$, so that by definition, if
$$G_t(z)=(1-|z|^2)^t R^{n+\alpha,t}g(z),$$ then $G_t \in T^p_\infty(\Bn)$. By formula \eqref{FracInt} and the standard properties of the fractional differential operators, we have
$$\frac{c(n,n+\alpha+t)}{c(n,n+\alpha)}P_{n+\alpha} G_t= R_{n+\alpha,t}R^{n+\alpha,t} g=g,$$
and we have shown the surjectivity.
\end{proof}

Before we prove duality, we will need a standard approximation result for the holomorphic tent spaces. Let $r \in (0,1)$ and define the dilation of $f$ by $f_r(z)=f(rz)$.

\begin{lemma}\label{approx}
Let $0<p<\infty$, $0<q\leq\infty$, $\alpha>-n-1$ and $f \in \mathcal{HT}^p_{q,\alpha}(\Bn)$. Then $f_r\to f$ in norm as $r\to 1^-$ and $\|f_r\|_{\mathcal{HT}^p_{q,\alpha}(\Bn)}\lesssim \|f\|_{\mathcal{HT}^p_{q,\alpha}(\Bn)}$.
\end{lemma}

\begin{proof}
We note that if $q=\infty$, this is a classical result about Hardy spaces. So, we will assume that $0<q<\infty$. Of course, this result is also known for $q=2$.

Suppose first that $\alpha=0$. We may assume that the aperture $\gamma$ on $\Gamma(\zeta)=\Gamma_\gamma(\zeta)$ satisfies $\gamma\geq 2$. Then if $z \in \Gamma(\zeta)$ and $r \in (0,1)$, we have
\begin{align*}
|1-\langle rz,\zeta\rangle|&\leq (1-r)+r|1-\langle z,\zeta\rangle|<(1-r)+r\frac{\gamma}{2}(1-|z|^2)\\
&=\left(1-\frac{\gamma}{2}\right)(1-r)+\frac{\gamma}{2}(1-r|z|^2)\leq \frac{\gamma}{2}(1-r^2|z|^2).
\end{align*}

We discover that $rz \in \Gamma(\zeta)$, so $r\Gamma(\zeta)\subset \Gamma(\zeta)$. This allows us to use a change of variables to see that
$$\int_{\Gamma(\zeta)}|f_r(z)|^q dV_n(z)\lesssim \int_{r\Gamma(\zeta)}|f(z)|^q dV_n(z)\leq  \int_{\Gamma(\zeta)}|f(z)|^q dV_n(z).$$
This already gives the desired norm bound.

If $f \in \mathcal{HT}^p_{q,\alpha}(\Bn)$, it follows that for every $r$, the function
$$\left(\int_{\Gamma(\zeta)}|f_r(z)-f(z)|^q dV_n(z)\right)^{p/q}$$
belongs to $L^1(\Sn)$ and has an integrable majorant. For almost every $\zeta \in \Sn$, we have
$$\int_{\Gamma(\zeta)}|f(z)|^q dV_n(z)<\infty,$$
and for such $\zeta$ the function $|f-f_r|^q$ has an integrable majorant over $\Gamma(\zeta)$. Since $f_r\to f$ pointwise as $r\to 1^-$, we see by the Lebesgue dominated convergence theorem that for almost every $\zeta$
$$\int_{\Gamma(\zeta)}|f_r(z)-f(z)|^q dV_n(z)\to 0.$$
But now another use of dominated convergence shows us that $f_r\to f$ in $\mathcal{HT}^p_{q,\alpha}(\Bn)$ as $r\to 1^-$.

Suppose next that $\alpha>-n-1$, $\alpha\neq 0$. Let $s$ be chosen so that neither $s+n$ nor $s+n-\alpha/q$ is a negative integer. By Theorem \ref{HTchar}, we have for any holomorphic $f$
$$\|f\|_{\mathcal{HT}^p_{q,\alpha}(\Bn)}\asymp \|R^{s,-\alpha/q}f\|_{\mathcal{HT}^p_{q,0}(\Bn)}.$$

Now, since the dilation $f \mapsto f_r$ commutes with the fractional derivatives (both being coefficient multipliers), we have
$$R^{s,-\alpha/q}f_r=(R^{s,-\alpha/q}f)_r\to R^{s,-\alpha/q}f$$
as $r\to 1^-$. This proves the remaining case.
\end{proof}

\begin{coro}
Let $0<p<\infty$, $0<q\leq\infty$, $\alpha>-n-1$ and $f \in \mathcal{HT}^p_{q,\alpha}(\Bn)$. Polynomials are dense in $\mathcal{HT}^p_{q,\alpha}(\Bn)$. Moreover, these spaces are separable.
\end{coro}

\begin{proof}
This is clear, since the functions $f_r$ can be approximated by polynomials uniformly on $\overline{\Bn}$. The separability claim then easily follows from the density of polynomials.
\end{proof}

\begin{theorem}\label{duality}
Let $1<p,q<\infty$ and $\alpha>-n-1$. The dual of $\mathcal{HT}^p_{q,\alpha}(\Bn)$ equals $\mathcal{HT}^{p'}_{q',\alpha}(\Bn)$ under the pairing $\langle \cdot,\cdot\rangle_{n+\alpha}$. The dual of $\mathcal{HT}^p_{1,\alpha}(\Bn)$ can be identified with $\mathcal{BT}^{p'}(\Bn)$.
\end{theorem}

\begin{proof}
Let us first deal with $1<q<\infty$. First, by applying the formula \eqref{EqG} and H\"older's inequality two times, it is clear that every $g \in \mathcal{HT}^{p'}_{q',\alpha}(\Bn)$ induces an element in the dual of $\mathcal{HT}^p_{q,\alpha}(\Bn)$.

Now, let us take arbitrary $\tau \in (\mathcal{HT}^p_{q,\alpha}(\Bn))^*$. By, the Hahn-Banach extension theorem, it can be extended to $\widetilde{\tau} \in (T^p_{q,\alpha}(\Bn))^*$ without increasing its norm. By duality of the measurable tent spaces, there exists $g_\tau \in T^{p'}_{q',\alpha}(\Bn)$ so that if now $f \in \mathcal{HT}^p_{q,\alpha}(\Bn)$, we have
$$\tau(f)=\widetilde{\tau}(f)=\langle f,g_\tau \rangle_{n+\alpha}=\langle P_{n+\alpha}f,g_\tau \rangle_{n+\alpha}=\langle f,P_{n+\alpha}g_\tau \rangle_{n+\alpha}.$$

By Proposition \ref{projection}, $P_{n+\alpha}g_\tau=h_\tau \in \mathcal{HT}^{p'}_{q',\alpha}(\Bn)$, which shows that each functional in the dual of $\mathcal{HT}^p_{q,\alpha}(\Bn)$ can represented by an element of $\mathcal{HT}^{p'}_{q',\alpha}(\Bn)$. Moreover
\begin{align*}\|\tau\|_{(\mathcal{HT}^p_{q,\alpha}(\Bn))^*}\leq \|h_\tau\|_{\mathcal{HT}^{p'}_{q',\alpha}(\Bn)}&\leq \|P_{n+\alpha}\|_{T^{p'}_{q',\alpha}(\Bn)\to T^{p'}_{q',\alpha}(\Bn)}\|g_\tau\|_{T^{p'}_{q',\alpha}(\Bn)}\\
&=\|P_{n+\alpha}\|_{T^{p'}_{q',\alpha}(\Bn)\to T^{p'}_{q',\alpha}(\Bn)}\|\tau\|_{(\mathcal{HT}^p_{q,\alpha}(\Bn))^*},
\end{align*}
so we have the equivalence of norms with an explicit constant coming from the norm of the Bergman projection.

It remains to show that, under this duality, the space $\mathcal{HT}^p_{q,\alpha}(\Bn)$ separates the points of $\mathcal{HT}^{p'}_{q',\alpha}(\Bn)$. We remark that, being bounded functions, the reproducing kernels
$$B_z^{n+\alpha}(u)=\frac{c(n,n+\alpha)}{(1-\langle u,z\rangle)^{1+2n+\alpha}}$$ belong to $\mathcal{HT}^p_{q,\alpha}(\Bn)$. If the holomorphic function $h_\tau$ represents the zero functional, then it follows that
$$\langle B_z^{n+\alpha},h_\tau\rangle_{n+\alpha}=\overline{h_\tau(z)}=0$$ for every $z \in \Bn$. Consequently $h_\tau=0$ as a holomorphic function. The proof is now complete.

We now turn our attention to the case $q=1$. If $g \in \mathcal{BT}^{p'}(\Bn)$, then since $P_{n+\alpha}:T^{p'}_{\infty,\alpha}(\Bn)\to \mathcal{BT}^{p'}(\Bn)$ is surjective, we can use the open mapping theorem and find $G \in T^{p'}_{\infty,\alpha}(\Bn)$ with $P_{n+\alpha}G=g$ and $\|g\|_{\mathcal{BT}^{p'}(\Bn)}\gtrsim \|G\|_{T^{p'}_{\infty,\alpha}(\Bn)}$. Thus, if $f$ is a polynomial, we can use Fubini's theorem to obtain
$$|\langle f,g \rangle_{n+\alpha}|=|\langle f,P_{n+\alpha}G\rangle_{n+\alpha}|=|\langle f,G\rangle_{n+\alpha}| \lesssim \|f\|_{\mathcal{HT}^p_{1,\alpha}(\Bn)}\|g\|_{\mathcal{BT}^{p'}(\Bn)}.$$
Since polynomials are dense in $\mathcal{HT}^p_{1,\alpha}(\Bn)$, we see that $g$ induces an element in its dual.

Next, we take arbitrary $\tau \in (\mathcal{HT}^p_{1,\alpha}(\Bn))^*$. As in the previous part, the Hahn-Banach theorem and the description of the dual of $T^p_{1,\alpha}(\Bn)$ together give $g_\tau \in T^{p'}_\infty(\Bn)$. By letting $f_r(z)=f(rz)$ ($r \in (0,1)$) denote the dilatation of $f$, by Lemma \ref{approx}, we have that if $f \in \mathcal{HT}^p_{1,\alpha}(\Bn)$, then $f_r\to f$ in $\mathcal{HT}^p_{1,\alpha}(\Bn)$ and $\|f_r\|_{\mathcal{HT}^p_{1,\alpha}(\Bn)}\lesssim \|f\|_{\mathcal{HT}^p_{1,\alpha}(\Bn)}$. Now, we have
$$\tau(f)=\langle f,g_\tau\rangle_{n+\alpha}=\lim_{r\to 1^-} \langle f_r,g_\tau\rangle_{n+\alpha}=\lim_{r\to 1^-}\langle f_r,P_{n+\alpha}g_\tau\rangle_{n+\alpha}=\lim_{r\to 1^-}\langle f_r,h_\tau \rangle_{n+\alpha}.$$
Since the integrand of the first limit is in $L^1_{n+\alpha}(\Bn)$, this limit, and therefore the other limits exist. This already shows that $\tau$ can be represented by $P_{n+\alpha} g_\tau = h_\tau \in \mathcal{BT}^{p'}(\Bn)$. Since $f_r$ is holomorphic, we can take $t>0$ and proceed with $f_r=P_{n+\alpha+t}f_r$, and use duality, to arrive at
$$|\tau(f)|=\lim_{r\to 1^-}|\langle f_r,(1-|\cdot|^2)^t R^{n+\alpha,t}h_\tau \rangle_{n+\alpha}|\lesssim \|f\|_{\mathcal{HT}^p_{1,\alpha}(\Bn)}\|h_\tau\|_{\mathcal{BT}^{p'}(\Bn)}.$$

To verify the claim about separation of points, one can proceed exactly as in the previous case. We obtain the desired duality result. Then pairing should generally be understood as
$$\lim_{r\to 1^-}\langle f_r,g\rangle_{n+\alpha}.$$
\end{proof}

\section{The space $\mathcal{BT}^{p',0}(\Bn)$: predual of $\mathcal{HT}^{p}_{1,\alpha}(\Bn)$}

\noindent In this section, we show that the spaces $\mathcal{HT}^{p}_{1,\alpha}(\Bn)$ themselves are dual spaces. Their preduals are the spaces $\mathcal{BT}^{p',0}(\Bn)$ that will be defined below. It should not come as a surprise that many of the results in this section resemble those already proven for $\mathcal{BT}^{p'}(\Bn)$, and therefore many details will be omitted. The duality result itself, however, uses a technique that has not been used in the paper so far (although the technique is not new, see \cite{ZhuBn}), so a complete proof is provided.

For every $r \in (0,1)$ and $f \in T^p_\infty(\Bn)$, we define
$$\rho_r(f)=\chi_{\Bn \setminus r\Bn} f,$$
and say that $f \in T^{p,0}_{\infty}(\Bn)$, if $\rho_r(f)\to 0$ in $T^p_\infty(\Bn)$ as $r\to 1^-$. It is clear from the definition that $ T^{p,0}_{\infty}(\Bn)$ is the closure of compactly supported functions  $f \in T^p_\infty(\Bn)$ under $\|\cdot\|_{T^p_\infty(\Bn)}$, and that it contains the space $L^\infty_0(\Bn)$ consisting of essentially bounded functions $f$ with
$$\lim_{r\to 1^-}\operatorname{ess} \sup_{|z|>r}|f(z)|=0.$$

\begin{definition}
Let, $1<p<\infty$, $t>0$, and assume that neither $n+s$ nor $n+s+t$ is a negative integer. A holomorphic $f$ belongs to the space $\mathcal{BT}^{p,0}(\Bn)$, if 
\begin{equation}\label{BT0}
(1-|\cdot|^2)^t R^{s,t}f \in T^{p,0}_\infty(\Bn).
\end{equation}
Moreover, the expression
$$\|(1-|\cdot|^2)^tR^{s,t}f\|_{T^p_\infty(\Bn)}$$ is a norm on $\mathcal{BT}^{p,0}(\Bn)$.
\end{definition}

\begin{lemma}\label{approx2}
Let $1<p<\infty$ and $f \in \mathcal{BT}^{p,0}(\Bn)$. Then, $f_r \to f$ in $\mathcal{BT}^{p,0}(\Bn)$ and $\|f_r\|_{\mathcal{BT}^{p}(\Bn)}\lesssim \|f\|_{\mathcal{BT}^{p}(\Bn)}$. Moreover, polynomials are dense in $\mathcal{BT}^{p,0}(\Bn)$.
\end{lemma}

\begin{proof}
Since the dilation commutes with the fractional derivatives, $r\Gamma(\zeta)\subset \Gamma(\zeta)$ and $(1-r^2|z|^2)^t \geq (1-|z|^2)^t$, the norm inequality is obvious, since
$$R^{s,t}f(rz)(1-|z|^2)^t\leq R^{s,t}f(rz)(1-r^2|z|^2)^t.$$

Next, let $\varepsilon>0$. For $f \in \mathcal{BT}^{p,0}(\Bn)$ there exists $r_0 \in (0,1)$ so that
$$\|\rho_s((1-|\cdot|^2)^t R^{s,t}f)\|_{T^p_\infty(\Bn)}<\varepsilon/3$$ and
$$\|\rho_s((1-|\cdot|^2)^t R^{s,t}f_r)\|_{T^p_\infty(\Bn)}<\varepsilon/3,$$
whenever $s,r>r_0$. Inside the compact set $\{|z|\leq r_0\}$, the functions involved are uniformly continuous, so the claim follows by letting $r\to 1^-$.

The density of polynomials follows in a standard way, since uniform convergence on $\Bn$ implies convergence in $\mathcal{BT}^{p}(\Bn)$.

\end{proof}

By a proof similar to that of Theorem \ref{BTchar}, one sees that every choice of $s$ and $t$ will provide an equivalent norm in an obvious way. Indeed, it only suffices to use the fact the polynomials are dense in $\mathcal{BT}^{p,0}(\Bn)$, along with the norm equivalence from Theorem \ref{BTchar}. It is also clear, after a similar chain of ideas, that the space $\mathcal{BT}^{p,0}(\Bn)$ is a closed subspace of $\mathcal{BT}^{p}(\Bn)$.

\begin{proposition}
Let $1<p<\infty$ and $\alpha>-n-1$. Then the Bergman projection $P_{n+\alpha}:T^{p,0}_\infty(\Bn)\to \mathcal{BT}^{p,0}(\Bn)$ is bounded and onto.
\end{proposition}

\begin{proof}
Clearly, the projection $P_{n+\alpha}$ maps compactly supported functions $f \in T^p_\infty(\Bn)$ inside $\mathcal{BT}^{p,0}(\Bn)$ with an appropriate norm bound coming from Proposition \ref{projection}. The claim now follows, since $\mathcal{BT}^{p,0}(\Bn)$ is closed in $\mathcal{BT}^{p}(\Bn)$.

One can extract a preimage of a function $f \in \mathcal{BT}^{p,0}(\Bn)$ in a manner identical to that in the proof of Proposition \ref{projection}; but now the preimage will be in $T^{p,0}_\infty(\Bn)$.
\end{proof}

We finally prove the duality $(\mathcal{BT}^{p,0}(\Bn))^*\sim \mathcal{HT}^{p'}_{1,\alpha}(\Bn)$. The argument of Theorem \ref{duality} does not seem to work, so we will have proceed in a different manner.

\begin{theorem}\label{duality2}
Let $1<p<\infty$. The dual of $\mathcal{BT}^{p,0}(\Bn)$ can be identified with $\mathcal{HT}^{p'}_{1,\alpha}(\Bn)$ under the pairing
$$\lim_{r\to 1^-} \langle f_r,g\rangle_{n+\alpha}.$$
\end{theorem}

\begin{proof}
It is already clear that every function in $\mathcal{HT}^{p'}_{1,\alpha}(\Bn)$ induces an element of the dual of $\mathcal{BT}^{p,0}(\Bn)$.

Let us now assume that $\tau \in (\mathcal{BT}^{p,0}(\Bn))^*$. With $m=(m_1, m_2 \dots, m_n)$ being a multi-index, it is well-known that if $f,g \in A^2_{n+\alpha}(\Bn)$, with series representations
$$f(z)=\sum_m a_m z^m, \qquad g(z)=\sum_m b_mz^m$$ then
$$c(n,n+\alpha)\langle f,g\rangle_{n+\alpha}=\sum_m \frac{m!\Gamma(2n+1+\alpha)}{\Gamma(2n+1+\alpha+|m|)}a_m\overline{b_m}.$$
Suppose now that $g$ has a representation as above with
$$b_m=\frac{\Gamma(2n+1+\alpha+|m|)}{m!\Gamma(2n+1+\alpha)} \overline{\tau(z^m)}.$$
Since the $\mathcal{BT}^{p}(\Bn)$ norm of $z^m$ is uniformly bounded with respect to the multi-index $m$, we know that the corresponding series defines a holomorphic function $g$ (see page 176 of \cite{ZhuBn}).

Now, if $f \in A^2_{n+\alpha}(\Bn)$
then, by the reproducing formula
$$f(z)=\langle f,B_z^{n+\alpha}\rangle_{n+\alpha}$$ with
$$B_z^{n+\alpha}(u)=\frac{c(n,n+\alpha)}{(1-\langle u,z\rangle)^{1+2n+\alpha}}=\sum_{m} \frac{\Gamma(2n+1+\alpha+|m|)}{m!\Gamma(2n+1+\alpha)}\overline{z}^m u^m,$$

we see that

$$\tau(f_r)=\sum_m \frac{m!\Gamma(2n+1+\alpha)}{\Gamma(2n+1+\alpha+|m|)} a_m \overline{b_m}r^{|m|}=c(n,n+\alpha)\langle f,g_r\rangle_{n+\alpha},$$
by first applying this formula to polynomials, and then approximating $f_r$ uniformly by them.

If $g$ were in $\mathcal{HT}^{p'}_{1,\alpha}(\Bn)$, we could immediately obtain
$$\tau(f_r)=\langle f_r,g\rangle_{n+\alpha},$$
but at the moment, we are not sure that the integral in question converges.
To show that this is the case, we use the duality
$$\mathcal{HT}^{p'}_{1,\alpha}(\Bn)\sim \mathcal{BT}^p(\Bn).$$
Note that, since $|g_r(z)|\to |g(z)|$ as $r\to 1^-$, Fatou's lemma shows that
$$\int_{\Gamma(\zeta)}|g(z)|(1-|z|^2)^\alpha dV_n(z)\leq \liminf_{r\to 1^-} \int_{\Gamma(\zeta)}|g_r(z)|(1-|z|^2)^\alpha dV_n(z),$$
so another application gives
$$\|g\|_{\mathcal{HT}^{p'}_{1,\alpha}(\Bn)}\lesssim \liminf_{r\to 1^-} \|g_r\|_{\mathcal{HT}^{p'}_{1,\alpha}(\Bn)}.$$

Bearing this in mind, it really suffices to show that the $\mathcal{HT}^{p'}_{1,\alpha}(\Bn)$ norms of the dilations $g_r$ are uniformly bounded. Now, by duality
$$\|g_r\|_{\mathcal{HT}^{p'}_{1,\alpha}(\Bn)}\asymp \sup\left\lbrace |\langle f,g_r\rangle_{n+\alpha}| : \|f\|_{\mathcal{BT}^{p}(\Bn)}\leq 1\right\rbrace.$$
But, by assumption
$$|\langle f,g_r\rangle_{n+\alpha}|=|\tau(f_r)|\lesssim \|f_r\|_{\mathcal{BT}^{p}(\Bn)}\lesssim \|f\|_{\mathcal{BT}^{p}(\Bn)},$$
where the last inequality follows from Lemma \ref{approx2}.

The rest of the proof follows as before; one can show the separation of points by applying the functional to the Bergman kernels.
\end{proof}

When $p=1=q$, the predual of $\mathcal{HT}^p_{q,\alpha}(\Bn)=A^1_{n+\alpha}(\Bn)$ is known to coincide with the little Bloch space $\mathcal{B}_0(\Bn)$ consisting of $f \in \mathcal{B}(\Bn)$ with
$$\lim_{|z|\to 1^-} (1-|z|^2)|\nabla f(z)|=0.$$

We remark that it is not difficult to see that 
$$\mathcal{B}_0(\Bn)\cup H^p(\Bn) \subset \mathcal{BT}^{p,0}(\Bn).$$
Also, the predual of $\mathcal{HT}^1_{\infty}(\Bn)$ is known to be isomorphic to space of holomorphic functions with vanishing mean oscillation $VMOA(\Bn)$.

We finally note that for every $\beta>-1$, one has
$$H^p(\Bn)\subset \mathcal{BT}^{p,0}(\Bn)\subset \mathcal{BT}^{p}(\Bn)\subset A^p_\beta (\Bn).$$

\section{Duality: the case $1<p<\infty$, $0<q<1$}

\noindent We will now deal with the case of small exponents $q$. For the corresponding Bergman spaces, there exists a very simple argument, see \cite{PR} and \cite{ZhuSmall}. However, this method seems to rely heavily on the fact the Bloch space is defined in terms of an $L^\infty$ norm. We see no obvious way to carry out this argument for the spaces $\mathcal{BT}^p(\Bn)$. Sometimes it helps to use some discretization technique, and this is what we will do. In this section we also need some results regarding atomic decomposition. These results are arguably known, and can be deduced from existing literature. To remain self-contained, we provide the proofs in the extent that will be need in Lemmas \ref{lattice} and \ref{converse}.

A sequence of points $\{z_ j\}\subset \Bn$ is said to be separated if there exists $\delta>0$ such that $\beta(z_ i,z_ j)\ge \delta$ for all $i$ and $j$ with $i\neq j$, where $\beta(z,u)$ denotes the Bergman metric on $\Bn$. This implies that there is $r>0$ such that the Bergman metric balls $D_ j=\{z\in \Bn :\beta(z,z_ j)<r\}$ are pairwise disjoint. \\

We need a well-known
result on decomposition of the unit ball $\Bn$.
By Theorem 2.23 in \cite{ZhuBn},
there exists a positive integer $N$ such that for any $0<r<1$ we can
find a sequence $\{a_k\}$ in $\Bn$ with the following properties:
\begin{itemize}
\item[(i)] $\Bn=\cup_{k}D(a_k,r)$.
\item[(ii)] The sets $D(a_k,r/4)$ are mutually disjoint.
\item[(iii)] Each point $z\in\Bn$ belongs to at most $N$ of the sets $D(a_k,4r)$.
\end{itemize}

Any sequence $\{a_k\}$ satisfying the above conditions  is called
an $r$-\emph{lattice}
(in the Bergman metric). Obviously any $r$-lattice is a separated sequence.

Suppose now that an $r$-lattice $Z=\{a_k\}$ is fixed, and consider the complex-valued sequences enumerated by this lattice: $c_k=f(a_k)$. For $0<p<\infty$ and $0<q\leq \infty$, the tent space $T^p_q(Z)$ consists of those sequences $\{c_k\}$ that satisfy
$$\|\{c_k\}\|_{T^p_q(Z)}=\left( \int_{\Sn} \left(\sum_{a_k \in \Gamma(\zeta)}|c_k|^q\right)^{p/q}d\sigma(\zeta)\right)^{1/p}<\infty.$$
Analogously, the tent space $T_\infty^p(Z)$ consists of $\{c_k\}$ with
$$\|\{c_k\}\|_{T^p_\infty(Z)}=\left( \int_{\Sn} \sup_{a_k \in \Gamma(\zeta)}|c_k|^pd\sigma(\zeta)\right)^{1/p}<\infty.$$

The following it the key duality result for such tent spaces of sequences. It is Proposition 2 in \cite{Ars}. See also \cite{Jev, Lue1, Lue2}. We will only need it for the case $0<q<1$.

\begin{otherth}\label{lattice}
Let $Z=\{a_k\}$ be an $r$-lattice. Suppose that $1<p,q<\infty$. Then the dual of $T^p_q(Z)$ can be identified with $T^{p'}_{q'}(Z)$ under the pairing
$$\langle \{c_k\},\{d_k\}\rangle_Z=\sum_{k=1}^\infty c_k\overline{d_k}(1-|a_k|^2)^n.$$
Moreover, if $0<q<1<p<\infty$, then the dual of $T^p_q(Z)$ can be identified with $T^{p'}_\infty(Z)$ under the same pairing.
\end{otherth}

\begin{lemma}\label{seq}
Let $0<p,q<\infty$, $Z=\{a_k\}$ be an $r$-lattice and $\theta>n\max(1,q/p,1/p,1/q)$. The mapping
$$S^\theta_Z \{c_k\}(z)=\sum_{k=1}^\infty c_k \frac{(1-|a_k|^2)^\theta}{(1-\langle z,a_k\rangle)^{\theta+(n+1+\alpha)/q}}$$
is bounded $T^p_q(Z)\to \mathcal{HT}^p_{q,\alpha}(\Bn)$.
\end{lemma}

\begin{proof}
This result is proven in Proposition 1 on page 350 of \cite{Lue2} in the setting a the upper half-space, and the case $q=2$ for the ball can be deduced from the corresponding result in \cite{P1}. There should be no doubt about the validity of the result. Since especially the case $0<q\leq 1$ is very easy, and we are only going to use it, we will prove this particular case.

Note that if $0<q\leq 1$, then
$$\left|\sum_{k=1}^\infty c_k \frac{(1-|a_k|^2)^\theta}{(1-\langle z,a_k\rangle)^{\theta+(n+1+\alpha)/q}}\right|\leq \sum_{k=1}^\infty |c_k|^q \frac{(1-|a_k|^2)^{q\theta}}{|1-\langle z,a_k\rangle|^{q\theta+(n+1+\alpha)}}.$$
If $\zeta \in \Sn$ and $\lambda>0$ large enough, we may now estimate by using Lemma \ref{FRgeneral}
\begin{align*}
&\int_{\Gamma(\zeta)} S^\theta_Z\{c_k\}(z)(1-|z|^2)^\alpha dV_n(z)\\\leq &\sum_{k=1}^\infty |c_k|^q \int_{\Gamma(\zeta)}\frac{(1-|a_k|^2)^{q\theta}}{|1-\langle z,a_k\rangle|^{q\theta+(n+1+\alpha)}}dV_n(z)\\
\lesssim &\sum_{k=1}^\infty |c_k|^q \int_{\Bn} \frac{(1-|z|^2)^{\lambda+\alpha}}{|1-\langle z,\zeta\rangle|^\lambda}\frac{(1-|a_k|^2)^{q\theta}}{|1-\langle z,a_k\rangle|^{q\theta+(n+1+\alpha)}}dV_n(z)\\
\lesssim& \sum_{k=1}^\infty |c_k|^q \frac{(1-|a_k|^2)^{q\theta}}{|1-\langle a_k,\zeta\rangle|^{q\theta}}.
\end{align*}
Therefore, by using Lemma \ref{Gamma} with $d\mu=\sum |c_k|^q d\delta_{a_k}$, where $\delta_{a_k}$ is the Dirac delta point mass, we obtain
\begin{align*}
&\int_{\Sn} \left(\int_{\Gamma(\zeta)} S^\theta_Z\{c_k\}(z)(1-|z|^2)^\alpha dV_n(z)\right)^{p/q}d\sigma(\zeta)\\
\lesssim &\int_{\Sn} \left(\sum_{k=1}^\infty |c_k|^q \frac{(1-|a_k|^2)^{q\theta}}{|1-\langle a_k,\zeta\rangle|^{q\theta}}\right)^{p/q}d\sigma(\zeta)\\
\lesssim &\int_{\Sn} \left(\sum_{a_k \in \Gamma(\zeta)} |c_k|^q \right)^{p/q}d\sigma(\zeta).
\end{align*}
The claim now follows.

\end{proof}

The following result is the tent space analogue of Lemma 3 of \cite{ZhuSmall}.

\begin{lemma}\label{embed}
Let $0<p<\infty$, $0<q\leq 1$, $\alpha>-n-1$ and $\alpha'=\alpha'(q,n,\alpha)=\alpha+(1/q-1)(n+1+\alpha)$. Then 
$$\mathcal{HT}^p_{q,\alpha}(\Bn)\subset \mathcal{HT}^p_{1,\alpha'}(\Bn)$$ with bounded inclusion.
\end{lemma}

\begin{proof}
For $\zeta \in \Sn$ and $r>0$, let us denote
$$\Gamma'(\zeta)=\bigcup_{z \in \Gamma(\zeta)}D(z,r).$$
It is known that the tent space $T^p_{q,\alpha}(\Bn)$ be defined also in terms of the approach regions $\Gamma'(\zeta)$. See, for instance \cite{Pav1, PR2015}.

Now, if $z \in \Gamma(\zeta)$, then for a holomorphic $f$, we have by subharmonicity
\begin{align*}
|f(z)|&\lesssim \frac{1}{(1-|z|^2)^{(n+1+\alpha)/q}}\left(\int_{D(z,r)}|f(u)|^q (1-|u|^2)^\alpha dV_n(u)\right)^{1/q}\\
&\lesssim \frac{1}{(1-|z|^2)^{(n+1+\alpha)/q}}\left(\int_{\Gamma'(\zeta)}|f(u)|^q (1-|u|^2)^\alpha dV_n(u)\right)^{1/q}.
\end{align*}
Writing $|f|=|f|^q |f|^{1-q}$ and applying this estimate to the second factor gives
\begin{align*}
&\int_{\Gamma(\zeta)}|f(z)|(1-|z|^2)^{\alpha'}dV_n(z)\\
\lesssim&\int_{\Gamma(\zeta)}|f(z)|^q (1-|z|^2)^\alpha \left(\int_{\Gamma'(\zeta)}|f(u)|^q (1-|u|^2)^\alpha dV_n(u)\right)^{1/q-1}dV_n(z)\\
\lesssim& \left(\int_{\Gamma'(\zeta)} |f(z)|^q (1-|z|^2)^\alpha dV_n(z)\right)^{1/q}.
\end{align*}
This obviously yields the result.
\end{proof}

We will prove a partial converse of Lemma \ref{seq}. Such result was already sketched for $q=2$ in \cite{Lue1} for the upper half-space, and is likely to be known, or at least expected to hold, by experts. We will only need the case $q=1$, and it is somewhat easier than the general case (especially when $q\leq 2n/(2n+1)$, the proof becomes more technical, see Section 2.5 of \cite{ZhuBn} and Theorem 4.33 of \cite{Zhu} for related discussion). Therefore, we provide a proof to remove any doubts the reader might have concerning it. Essentially the same proof works for $q>2n/(2n+1)$, for reasons that become obvious by reading the argument.

\begin{lemma}\label{converse}
Let $0<p<\infty$, $Z=\{a_k\}$ be an $r$-lattice and $\theta>n\max(1,q/p,1/p,1/q)$. If $r$ is chosen to be sufficiently small, and $f \in \mathcal{HT}^p_{1,\alpha}(\Bn)$, then there exists a sequence $\{c_k\}\in T^p_1(Z)$ so that
$$f(z)=S_Z^\theta\{c_k\}(z).$$
Moreover $\|\{c_k\}\|_{T^p_1(Z)}\lesssim \|f\|_{\mathcal{HT}^p_{1,\alpha}(\Bn)}$.

\end{lemma}

\begin{proof}
According to Lemma 2.28 in \cite{ZhuBn}, we may partition $\Bn$ into sets $D_k$ that are pairwise disjoint and
$$D(a_k,r/4)\subset D_k \subset D(a_k,r)$$ for every $k$.

Let us set
$$c_k=\int_{D_k}(1-|z|^2)^{\theta+\alpha}dV_n(z)f(a_k)(1-|a_k|^2)^{-\theta}.$$
Since $D_k\subset D(a_k,r)$, we use subharmonicity to see that
$$|c_k|\lesssim \int_{D(a_k,r)}|f(z)|(1-|z|^2)^{\alpha}dV_n(z).$$
So, summing over $a_k \in \Gamma(\zeta)$ yields, as in the proof of Lemma \ref{embed}
$$\sum_{a_k \in \Gamma(\zeta)}|c_k|\lesssim \int_{\Gamma'(\zeta)}|f(z)|(1-|z|^2)^{\alpha}dV_n(z).$$
It follows that $\|\{c_k\}\|_{T^p_1(Z)}\lesssim \|f\|_{\mathcal{HT}^p_{1,\alpha}(\Bn)}$.

On the other hand, it easy to see that (with our choice of $\{c_k\}$) Lemma 4.32 of \cite{Zhu} generalizes into $\Bn$ as
$$|f(z)-S_Z^\theta\{c_k\}(z)|\lesssim r^{2n+1}\sum_{k=1}^\infty \frac{(1-|a_k|^2)^\theta}{|1-\langle z,a_k\rangle|^{\theta+n+1+\alpha}}\int_{D(a_k,1)}|f(u)|(1-|u|^2)^\alpha dV_n(u).$$
The factor $r^{2n+1}$ comes from integrating $|z|$ over $D(0,r)$ -- compare with the proof of Lemma 4.31 in \cite{ZhuBn}.

This gives us by first using Lemma \ref{FRgeneral} and then Lemma \ref{Gamma} with the measure
$$d\mu=\sum_{k=1}^\infty \left(\int_{D(a_k,1)}|f(u)|(1-|u|^2)^\alpha dV_n(u)\right)\delta_{a_k},$$
\begin{align*}
&\int_{\Sn}\left(\int_{\Gamma(\zeta)}|f(z)-S_Z^\theta\{c_k\}(z)|(1-|z|^2)^\alpha dV_n(z)\right)^{p}d\sigma(\zeta)\\
\lesssim& r^{p(2n+1)}\int_{\Sn}\left(\sum_{k=1}^\infty\frac{(1-|a_k|^2)^\theta}{|1-\langle \zeta,a_k\rangle|^\theta}\int_{D(a_k,1)}|f(u)|(1-|u|^2)^\alpha dV_n(u)\right)^p d\sigma(\zeta)\\
\lesssim& r^{p(2n+1)}\int_{\Sn}\left(\sum_{a_k \in \Gamma(\zeta)}\int_{D(a_k,1)}|f(u)|(1-|u|^2)^\alpha dV_n(u)\right)^pd\sigma(\zeta).
\end{align*}
Now, we may assume that there exists $C>0$ so that every point in $\Bn$ belongs to at most $C/r^{2n}$ of the sets $D(a_k,1)$ (see Lemmas 4.6 and 4.7 of \cite{Zhu}; the proof is presented for $n=1$, but generalizes to $n>1$ in an obvious way by noting that if $\beta(u,z)<R$, then $V_n(D(z,r))\geq C_R r^{2n}V_n(D(u,1))$, where $C_R$ only depends on $R>0$), and therefore
\begin{align*}
&\int_{\Sn}\left(\int_{\Gamma(\zeta)}|f(z)-S_Z^\theta\{c_k\}(z)|(1-|z|^2)^\alpha dV_n(z)\right)^p d\sigma(\zeta) \\
\lesssim r^p &\int_{\Sn}\left(\int_{\Gamma'(\zeta)}|f(z)|(1-|z|^2)^\alpha dV_n(z)\right)^p d\sigma(\zeta),
\end{align*}
where the implicit constants do not depend on $r$. If follows that if $T_Z^\theta$ is the linear operator, which for our choice of $\{c_k\}$ satisfies
$T_Z^\theta f=\{c_k\}$, then if $r>0$ is small enough, $$I-S_Z^\theta T_Z^\theta:\mathcal{HT}^p_{1,\alpha}(\Bn)\to\mathcal{HT}^p_{1,\alpha}(\Bn)$$ has norm smaller than one. Hence $S_Z^\theta T_Z^\theta$ is invertible with inverse given as the Neumann series. 

We already saw that $\{c_k\} \in T^p_1(Z)$, so the proof is complete.
\end{proof}

We are now ready to prove our last main result.

\begin{theorem}\label{duality3}
Let $0<q<1<p<\infty$, $\alpha>-n-1$ and $\alpha'=\alpha+(1/q-1)(n+1+\alpha)$. Then the dual of $\mathcal{HT}^p_{q,\alpha}(\Bn)$ can be identified with $\mathcal{BT}^{p'}(\Bn)$ under the pairing
$$\lim_{r\to 1^-} \langle f_r,g\rangle_{n+\alpha'}.$$
\end{theorem}

\begin{proof}
Let us first take an arbitrary element $\tau \in (\mathcal{HT}^p_{q,\alpha}(\Bn))^*$. By Lemma \ref{embed}, if $f \in \mathcal{HT}^p_{q,\alpha}(\Bn)$, then
$$f(z)=c(n,n+\alpha')\int_{\Bn}\frac{(1-|u|^2)^{n+\alpha'}}{(1-\langle z,u\rangle)^{1+2n+\alpha'}}f(u)dV_n(u).$$
A reasoning similar to that in the proof of Theorem \ref{duality2} shows that there exists a holomorphic $g$ so that
$$\tau(f_r)=\langle f,g_r\rangle_{n+\alpha'}.$$
Now, suppose that $Z=\{a_k\}$ is an $r$-lattice satisfying the assumptions of Lemma \ref{converse}, $\{c_k\} \in T^p_q(Z)$ and $\theta>n\max(1,q/p,1/p,1/q)$. Then by Lemma \ref{seq}, we have
\begin{equation}\label{esti}
|\tau(S_Z^\theta\{c_k\})|\lesssim \|\tau\|_{(\mathcal{HT}^p_{q,\alpha}(\Bn))^*}\|\{c_k\}\|_{T^p_q(Z)}.
\end{equation}
But
\begin{align*}
\tau(S_Z^\theta\{c_k\})&=\sum_{k=1}^\infty c_k \int_{\Bn}\frac{(1-|a_k|^2)^\theta}{(1-\langle z,a_k\rangle)^{\theta+(n+1+\alpha)/q}}\overline{g_r(z)}(1-|z|^2)^{n+\alpha'}dV_n(z)\\
&=\sum_{k=1}^\infty c_k \overline{(1-|a_k|^2)^{\theta-n}R^{n+\alpha',\theta-n}g_r(a_k)}(1-|a_k|^2)^n.
\end{align*}
By Theorem \ref{lattice} and estimate \eqref{esti}, we obtain
\begin{equation}\label{discrete}
\int_{\Sn} \left(\sup_{a_k \in \Gamma(\zeta)}(1-|a_k|^2)^{\theta-n}|R^{n+\alpha',\theta-n}g_r(a_k)|\right)^{p'}d\sigma(\zeta)\lesssim \|\{c_k\}\|_{T^p_q(Z)}.
\end{equation}

Next, let $f$ be an arbitrary function in $\mathcal{HT}^p_{1,\alpha'}(\Bn)$. By Lemma \ref{converse}, we deduce that for every $\theta'>n$
$$f(z)=\sum_{k=1}^\infty d_k \frac{(1-|a_k|^2)^{\theta'}}{(1-\langle z,a_k\rangle)^{\theta'+(n+1+\alpha')}},$$
where $\{d_k\} \in T^p_1(Z)$ and $\|\{d_k\}\|_{T^p_1(Z)}\lesssim \|f\|_{\mathcal{HT}^p_{1,\alpha'}(\Bn)}$. Moreover, the series converges in $\mathcal{HT}^p_{1,\alpha'}(\Bn)$, so we may assume that $f$ is bounded. We apply $g$ to $f$ under the pairing $\langle \cdot,\cdot\rangle_{n+\alpha'}$ to obtain as before
$$
\langle f,g_r\rangle_{n+\alpha'}=\sum_{k=1}^\infty d_k \overline{(1-|a_k|^2)^{\theta'-n}R^{n+\alpha',\theta'-n}g_r(a_k)}(1-|a_k|^2)^n.
$$

Since we may choose $\theta'=\theta$, we let $M$ be the supremum in \eqref{discrete} to conclude that
$$|\langle f,g_r\rangle_{n+\alpha'}|\lesssim M\|\{d_k\}\|_{T^p_1(Z)}\lesssim M\|f\|_{\mathcal{HT}^p_{1,\alpha'}(\Bn)}.$$
Therefore, $g_r$ is a holomorphic function in the dual of $\mathcal{HT}^p_{1,\alpha'}(\Bn)$ under the pairing $\langle f,g_r\rangle_{n+\alpha'}$. By Theorem \ref{duality}, $g_r \in \mathcal{BT}^{p'}(\Bn)$ with a norm bound independent of $r$. It follows that $g \in \mathcal{BT}^{p'}(\Bn)$.

Conversely, let $g \in \mathcal{BT}^{p'}(\Bn)$. By Theorem \ref{duality} and Lemma \ref{embed}, if $f \in \mathcal{HT}^p_{q,\alpha}(\Bn)$, we have

$$|\langle f_r,g\rangle_{n+\alpha'}|\lesssim \|g\|_{\mathcal{BT}^{p'}(\Bn)}\|f\|_{\mathcal{HT}^p_{1,\alpha'}(\Bn)}\lesssim \|g\|_{\mathcal{BT}^{p'}(\Bn)}\|f\|_{\mathcal{HT}^p_{q,\alpha}(\Bn)}.$$
Therefore, $g$ represents an element in the dual.

Yet again, the separation of points follows easily by applying $g$ to the Bergman kernel functions.

\end{proof}

\section{Duality: the case $0<p\leq 1$, and the spaces $\mathcal{CT}_{q,\alpha}(\Bn)$}

We begin by defining the holomorphic function spaces $\mathcal{CT}_{q,\alpha}(\Bn)$.

\begin{definition}
Let $1<q<\infty$ and $\alpha>-n-1$. The space $\mathcal{CT}_{q,\alpha}(\Bn)$ consists of those holomorphic $f$ that belong to $T^\infty_{q,\alpha}(\Bn)$. We equip it with the norm
$$\|f\|_{\mathcal{CT}_{q,\alpha}(\Bn)}=\operatorname{ess}\sup_{\zeta \in \Sn} \left(\sup_{u \in \Gamma(\zeta)}\frac{1}{(1-|u|^2)^n}\int_{Q(u)}|f(z)|^q (1-|z|^2)^{n+\alpha}dV_n(z)\right)^{1/q}.$$
\end{definition}

It is clear that this norm is equivalent to the supremum from Lemma \ref{CM} for any $T>0$, when $d\mu(z)=|f(z)|^q (1-|z|^2)^{n+\alpha}$. Hence, the classical space $BMOA(\Bn)$ consists of holomorphic $f$, whose first order derivatives belong to $\mathcal{CT}_{2,1-n}(\Bn)$.

We will next characterize $\mathcal{CT}_{q,\alpha}(\Bn)$ in terms of all fractional derivatives $R^{s,t}$. The method will be similar to that of Theorem \ref{HTchar} and \ref{BTchar}, but we have to use formula \eqref{FR2} of Lemma \ref{FRgeneral}. We will show how the argument goes, but will henceforth avoid repeating it.

\begin{theorem}\label{CTchar}
Let $1<q<\infty$, $\alpha>-n-1$ and $s,t$ be real numbers. Assume further that neither $s+n$ nor $s+n+t$ is a negative integer and $qt+n+1+\alpha>0$. Then the following are equivalent.
\begin{itemize}
\item[(a)] $f \in \mathcal{CT}_{q,\alpha}(\Bn)$;
\item[(b)] $R^{s,t}f \in \mathcal{CT}_{q,\alpha+qt}(\Bn)$. 
\end{itemize}
Moreover, the corresponding norms are equivalent.
\end{theorem}

\begin{proof}
Let us assume (a), and let $t+(n+1+\alpha)/q>T>0$. Then, by reasoning as in the proof of Theorem \ref{HTchar}, for a sufficiently large positive integer $N$ we will use part \eqref{FR2} of Lemma \ref{FRgeneral} to obtain
\begin{align*}
&\int_{\Bn}\frac{(1-|a|^2)^T}{|1-\langle z,a\rangle|^{n+T}}(1-|z|^2)^{\alpha+tq+n}|R^{s,t}f(z)|^q dV_n(z)\\
\lesssim &\int_{\Bn}\frac{(1-|a|^2)^T}{|1-\langle z,a\rangle|^{n+T}}(1-|z|^2)^{t+(n+1+\alpha)/q-1}\int_{\Bn}\frac{(1-|u|^2)^{s+N}|f(u)|^q}{|1-\langle z,u\rangle|^{1+n+s+N+t-(n+1+\alpha)/q'}}dV_n(u)\\
\lesssim &\int_{\Bn}\frac{(1-|a|^2)^T}{|1-\langle z,u\rangle|^{n+T}}(1-|u|^2)^{n+\alpha}|f(u)|^q dV_n(u).
\end{align*}
We have shown that (a) implies (b). The converse statement is left as an exercise for the reader.

\end{proof}

The following corollary might be of some interest.

\begin{coro}
Assume that $t>0$ and neither $s+n$ nor $s+n+t$ is a negative integer. A holomorphic $f$ belongs to the space $BMOA(\Bn)$ if and only if
$$R^{s,t}f\in \mathcal{HT}^1_{2,2t-1-n}(\Bn).$$
\end{coro}

We now prove a central result on our way to the duality $(\mathcal{HT}^1_{q,\alpha}(\Bn))^*\sim \mathcal{CT}_{q',\alpha}(\Bn)$. The core idea of the proof is similar to that of Theorem 3.2 in \cite{BP}.

\begin{proposition}\label{projection2}
Let $1<q<\infty$ and $\alpha>-n-1$. The Bergman projection $$P_{n+\alpha}:T^\infty_{q,\alpha}(\Bn)\to \mathcal{CT}_{q,\alpha}(\Bn)$$ is bounded and onto.
\end{proposition}

\begin{proof}
Suppose that $f \in T^\infty_{q,\alpha}(\Bn)$. This means that
$$d\mu(z)=|f(z)|^q (1-|z|^2)^{n+\alpha}dV_n(z)$$ is a Carleson measure.

Let now $\varepsilon>$ be chosen so that $n+1+\alpha-q'\varepsilon>0$, $n+1+\alpha-q\varepsilon>0$ and let $T>0$ be chosen so that $n+1+\alpha-q\varepsilon >T$.

Then, 
\begin{align*}
&\int_{\Bn}\frac{(1-|a|^2)^T}{|1-\langle a,z\rangle|^{n+T}}(1-|z|^2)^{n+\alpha}|P_{n+\alpha}f(z)|^q dV_n(z)\\
\lesssim &\int_{\Bn}\frac{(1-|a|^2)^T}{|1-\langle a,z\rangle|^{n+T}}(1-|z|^2)^{n+\alpha-q\varepsilon}\left(\int_{\Bn}\frac{(1-|u|^2)^{n+\alpha+q\varepsilon}}{|1-\langle z,u\rangle|^{1+2n+\alpha}}|f(u)|^q dV_n(u)\right)dV_n(z)
\end{align*}
Notice that by our choice $n+T<n+(n+1+\alpha)-q\varepsilon$, but $1+2n+\alpha>n+\alpha-q\varepsilon+n+1$. By the formula \eqref{FR2} of Lemma \ref{FRgeneral} along with Fubini's theorem, we estimate this to be dominated by
$$\int_{\Bn}\frac{(1-|a|^2)^{T}}{|1-\langle a,u\rangle|^{n+T}}(1-|u|^2)^{n+\alpha}|f(u)|^q dV_n(u).$$

Since $f \in T^\infty_{q,\alpha}(\Bn)$, we are done. Obviously, the operator is surjective.
\end{proof}

Recall that the argument in the proof of Proposition \ref{projection} essentially depended on the fact that $p\geq q>1$, and the other case was deduced by duality. Since we now have covered the dual case for $T^1_{q,\alpha}(\Bn)$, we obtain the following corollary. Note that the statement is false when also $q=1$, since it is well-known that $P_{n+\alpha}:L^1_{n+\alpha}(\Bn)\to A^1_{n+\alpha}(\Bn)$ is not bounded.

\begin{coro}
Let $\alpha>-n-1$ and $1<q<\infty$. Then the Bergman projection
$$P_{n+\alpha}:T^1_{q,\alpha}(\Bn)\to \mathcal{HT}^1_{q,\alpha}(\Bn)$$
is bounded and onto.
\end{coro}

We have now sufficient information about the behaviour of the spaces $\mathcal{CT}_{q,\alpha}(\Bn)$, as well as the Bergman projection $P_{n+\alpha}$. In order to get a complete picture of duality, we still need the following lemma.

\begin{lemma}\label{embed2}
If $p<1$, $0<q<\infty$ and $\alpha>-n-1$, then
$$\mathcal{HT}^p_{q,\alpha}(\Bn)\subset A^1_{(1/p-1)n+(n+1+\alpha)/q-1}(\Bn)$$
with bounded inclusion.
\end{lemma}

\begin{proof}
Arguably one could obtain this result as a special case of Theorem 3 of \cite{Lue2}, once the passage from the upper half-plane of $\mathbb{C}$ to the unit ball of $\mathbb{C}^n$ is made.  We show here, how this particular case easily follows from the well-known embeddings of the Hardy space.

First, recall that if $p<s<\infty$, then
$$H^p(\Bn)\subset A^s_{(s/p-1)n-1}(\Bn)$$ with bounded inclusion.

Applying this to a suitable fractional differential operator $R^{s,(n+1+\alpha)/2}$ yields according to Theorem \ref{area}
$$\mathcal{HT}^p_{2,\alpha}(\Bn)\subset A^s_{(s/p-1)n-1+(s/2)(n+1+\alpha)}.$$
Next, if $k$ is a natural number, then since $f \in \mathcal{HT}^p_{2k,\alpha}(\Bn)$ if and only if $f^k \in \mathcal{HT}^{p/k}_{2,\alpha}(\Bn)$, we obtain
\begin{equation}\label{eq1}
\mathcal{HT}^p_{2k,\alpha}(\Bn)\subset A^s_{(s/p-1)n-1+(s/(2k))(n+1+\alpha)}.
\end{equation}
Finally, pick $k$ big enough so that $2k>q$. A slight modification of Lemma \ref{embed} gives
\begin{equation}\label{eq2}
\mathcal{HT}^p_{q,\alpha}(\Bn)\subset \mathcal{HT}^p_{2k,\alpha+(2k/q-1)(n+1+\alpha)}(\Bn).
\end{equation}
Combining the embeddings \eqref{eq1} and \eqref{eq2} proves the lemma after matching all the parameters.

\end{proof}

The following Theorem \ref{duality4} is our last main result. It is the missing piece in the picture of duality.

\begin{theorem}\label{duality4}
Let $0<p<1$, $0<q<\infty$ and $\alpha>-1$. Then the dual of $\mathcal{HT}^p_{q,\alpha}(\Bn)$ can be identified with the Bloch space $\mathcal{B}(\Bn)$ under the pairing
$$\lim_{r\to 1^-}\langle f_r,g\rangle_{(1/p-1)n+(n+1+\alpha)/q-1}.$$
If $p=1$ and $0<q\leq 1$ then the above claim still holds. If $1<q<\infty$, then the dual of $\mathcal{HT}^1_{q,\alpha}(\Bn)$ can be identified with $\mathcal{CT}_{q',\alpha}(\Bn)$ under the pairing $\langle \cdot,\cdot\rangle_{n+\alpha}$.
\end{theorem}

\begin{proof}
Suppose first that $0<p<1$ and $1<q<\infty$ or $p=1$ and $q\leq 1$. Set $\beta=(1/p-1)n+(n+1+\alpha)/q-1$. Notice that for $\theta$ large enough, the functions
$$F_a(z)=\frac{(1-|a|^2)^{\theta}}{(1-\langle z,a\rangle)^{\theta+(n+1+\alpha)/q+n/p}}$$ have $\mathcal{HT}^p_{q,\alpha}(\Bn)$ norms bounded from above.

As before, it can be shown that $\tau\in (\mathcal{HT}^p_{q,\alpha}(\Bn))^*$ can be represented by a holomorphic $g$ under the pairing 
$$\lim_{r\to 1^-}\langle f_r,g\rangle_{(1/p-1)n+(n+1+\alpha)/q-1}.$$
Putting $F_a$ to this expression and using duality, one obtains
$$|R^{(1/p-1)n+(n+1+\alpha)/q-1,\theta}g(a)|\lesssim (1-|a|^2)^{-\theta},$$
so $g \in \mathcal{B}(\Bn)$.

Conversely, by Lemma \ref{embed2}, it is obvious that Bloch functions generate bounded linear functionals $\mathcal{HT}^p_{q,\alpha}(\Bn)\to \mathbb{C}$.\\

The proof of the case $p=1$ and $1<q<\infty$ is goes exactly like the proof of Theorem \ref{duality}, since we now have Proposition \ref{projection2} at our disposal.
\end{proof}

It is not surprising that if we define the spaces $\mathcal{CT}^0_{q,\alpha}(\Bn)$ to consist of those holomorphic $f$ such that the measure
$$d\mu(z)=|f(z)|^q (1-|z|^2)^{n+\alpha}dV_n(z)$$ is a vanishing Carleson measure, we can use the methodology presented here to obtain the following corollary.

\begin{coro}
Let $1<q<\infty$ and $\alpha>-n-1$. Then the dual of $\mathcal{CT}^0_{q,\alpha}(\Bn)$ is $\mathcal{HT}^1_{q',\alpha}(\Bn)$ under the pairing $\langle \cdot,\cdot\rangle_{n+\alpha}$.
\end{coro}

Note that the duality in all the cases considered in this paper could be defined in terms of several other pairings. For instance, using Theorems \ref{area}, \ref{HTchar}, \ref{BTchar} and \ref{CTchar} along with the following identity (which can be proven by comparing Taylor series expansions or by Fubini's theorem, see \cite{ZZ}),
$$\langle f,g\rangle_{n+\alpha}=\operatorname{const}\langle R^{n+\alpha,t}f,R^{n+\alpha+t,t}g\rangle_{n+\alpha+2t},$$
allows one to include the classical Hardy spaces to the scale. This fact can also be used to get rid of the, perhaps unwanted, limit process in some of the dualities discussed in this work.

Since $R^{s,t}H^1(\Bn)=\mathcal{HT}^1_{2,2t-1-n}(\Bn)$, we can recover the celebrated dualities
$$H^1(\Bn)^*\sim BMOA(\Bn)\quad \text{ and } \quad VMOA(\Bn)^*\sim H^1(\Bn).$$

\section{Further remarks}

\noindent In this section we list some topics of further discussion and research based on our subjective level of interest. \\

\subsection{Other derivatives} 

We remark that an analogous characterization of all the spaces studied here in terms of other differential operators is also possible. For instance, if $R$ is the radial derivative operator (see \cite{ZZ}, \cite{ZhuBn}), then a characterization can be obtained in terms of $R^k$ for any positive integer $k$. The proof will be very similar to that Theorems \ref{HTchar},\ref{BTchar} and \ref{CTchar}, and is omitted here.

For our purposes, the operators $R^{s,t}$ have turned out to be practical. First, because of their generality; and second, because they are injective, so they can define a norm without introducing any correction terms (such as point-evaluations at zero).

\subsection{Closure of $H^p(\Bn)$ in $\mathcal{BT}^p(\Bn)$}

It might be of interest that the closure of $H^p(\Bn)$ in $\mathcal{BT}^{p}(\Bn)$ equals $\mathcal{BT}^{p,0}(\Bn)$. This is straightforward, since polynomials are dense in both $H^p(\Bn)$ and $\mathcal{BT}^{p,0}(\Bn)$, and we have the estimate $\|f\|_{H^p(\Bn)}\gtrsim \|f\|_{\mathcal{BT}^{p,0}(\Bn)},$ valid for every polynomial $f$.

This could be compared with the problem of J. M. Anderson, J. Clunie and C. Pommerenke \cite{ACP} on the description of the closure of bounded analytic functions in the Bloch norm. 

We mention that the closure of $H^p(\Bn)\cap \mathcal{B}(\Bn)$ in $\mathcal{B}(\Bn)$ has been recently studied. See \cite{GMP} and \cite{MN}.\\

\subsection{Bergman projection for $0<p<1$}

It is clear from the proof of Proposition \ref{projection} that the Bergman projection $P_{n+\alpha}$ is bounded on $T^p_{q,\alpha}(\Bn)$ even when $0<p<1$, if $q$ and $\alpha$ are chosen appropriately. It might be of interest to give complete characterization for boundedness in this case well. The proof can be quite technical, and does no benefit to the present work, so we leave this problem open for later research.\\

\subsection{Norm of the Bergman projection}

In \cite{Liu} C. Liu calculated the best constants involved in the Forelli-Rudin estimates, such as Lemma in \ref{IctBn}. This allowed him to obtain the most precise information on the $L^p$ norm of the Bergman projection so far. This result was later extended to cover the standard weights in a joint paper \cite{LPZ} of Liu, L. Zhou and the author.

In a similar philosophy, in order to obtain an effective upper bound for the norm $\|P_{n+\alpha}\|_{T^{p}_{q}(\Bn)\to T^{p}_{q}(\Bn)}$, one should look for the sharp constants in Lemma \ref{Gamma} and \ref{FRgeneral}. Completely finding the norm would mean also taking into account the aperture $\gamma$ of the approach regions $\Gamma(\zeta)=\Gamma_\gamma(\zeta)$, and this might seem hopeless. Asymptotics in terms of $p,q,n$ and $\alpha$ seem possible.\\

\subsection{Weighted Bergman spaces} In \cite{PR2015} J. \'A. Pel\'aez and J. R\"atty\"a have introduced weighted Bergman type tent spaces on the unit disk. Their work allows one to replace $(1-|z|^2)^\alpha$ by a more general radial weight function $\omega(z)$, subject to a certain doubling condition. 

One would expect similar results to hold in this setting (with possibly some less general weights). Perhaps the biggest obstacle could be to extract a sharp enough analogue of Lemma \ref{Gamma}. Such result is contained in \cite{PR2015}, but naturally the requirements on the parameters are much less obvious than here. Admittedly, this topic could be very interesting considering that these spaces are known to exhibit some transition between the Hardy and Bergman spaces.\\


\end{document}